\documentclass[11pt,reqno]{amsart}
\usepackage{color}
\usepackage[active]{srcltx}
\usepackage{a4wide}
\usepackage{amssymb, amsmath, mathtools}
\usepackage{graphicx}
\usepackage{float}
\usepackage[active]{srcltx}
\usepackage[ansinew]{inputenc}
\usepackage{hyperref}

\theoremstyle{plain}% default
\newtheorem{theorem}{Theorem}[section]
\newtheorem{maintheorem}{Theorem}
\newtheorem{lemma}[theorem]{Lemma}
\newtheorem{proposition}[theorem]{Proposition}
\newtheorem{corollary}[theorem]{Corollary}

\theoremstyle{remark}

\newtheorem{definition}{Definition}

\newtheorem{remark}[theorem]{Remark}%[section]

\numberwithin{equation}{section}

\newcommand{\NN}{{\mathbb{N}}}
\newcommand{\ZZ}{{\mathbb{Z}}}
\newcommand{\RR}{{\mathbb{R}}}
\newcommand{\EU}{{\mathbb{S}}}

\newcommand{\In}{{\text{In}}}
\newcommand{\Out}{{\text{Out}}}

\newcommand{\loc}{{\text{loc}}}

\newcommand{\dpt}{\displaystyle}

\begin{document}

\title[Abundance of infinite switching]{Abundance of infinite switching}
 
 \author[Alexandre A. P. Rodrigues]{Alexandre A. P. Rodrigues \\ Center of Mathematics, Sciences Faculty, University of Porto \\ 
and Lisbon School of Economics and Management \\
Rua do Campo Alegre 687, 4169--007 Porto, Portugal\\ \\}
\address[A. A. P. Rodrigues]{ Centro de Matem\'atica\\
 Faculdade de Ci\^encias, Universidade do Porto\\ 
Rua do Campo Alegre 687, 4169--007 Porto, Portugal\\
Lisbon School of Economics and Management \\
Rua do Quelhas 6, 1200-781 Lisboa, Portugal}
\email[A.A.P.Rodrigues]{alexandre.rodrigues@fc.up.pt}

\bigbreak

\author[Lu\'isa Castro]{\\ \\ Lu\'isa Castro \\ Center for Health Technology and Services Research - CINTESIS, \\ University of Porto, \\
Rua Dr. Pl\'acido da Costa, 4200-450 Porto, Portugal \\ \\ } 
\address[L. Castro]{Center for Health Technology and Services Research - CINTESIS, University of Porto, Rua Dr. Pl\'acido da Costa, 4200-450 Porto, Portugal}
\email[L. Castro]{luisacastro@med.up.pt}

\bigbreak

\date{\today}

\thanks{AR was partially supported by CMUP (UIBD/MAT/00144/2020), which is funded by Funda\c{c}\~ao para a Ci\^encia e a Tecnologia (FCT) with national  and European structural funds through the programs FEDER, under the partnership agreement PT2020. AR also benefits from the grant CEECIND/01075/2020 of the Stimulus of Scientific Employment -- 3rd Edition (Individual Support)  awarded by FCT}


\begin{abstract}
In this article, we describe a class of vector fields exhibiting \emph{abundant switching}  near a network: for every neighbourhood of the network and every infinite admissible path, the set of initial conditions within the neighbourhood that follows the path has positive Lebesgue measure. 
 
The proof relies on the existence of ``large'' strange attractors in the terminology of Broer, Sim\'o and Tatjer (Nonlinearity, 667--770, 1998) near a heteroclinic tangle unfolding an attracting network with a two-dimensional heteroclinic connection. For our class of vector fields,  any small non-empty open ball of initial conditions realizes infinite switching. We illustrate the theory with a specific one-parameter family of differential equations, for which we are able to characterise its global dynamics for almost all parameters. 
 
 \end{abstract}

\maketitle
\setcounter{tocdepth}{1}
%\tableofcontents

\section{Introduction}\label{intro}
A \emph{heteroclinic cycle} is a set of finitely many invariant saddles and trajectories connecting them in a cyclic way. A connected union of finitely many heteroclinic cycles is what we call a \emph{heteroclinic network}. These structures are associated with intermittent behaviour and are used to model intermittency dynamics in several applications, including neuroscience, nonlinear oscillations, geophysics, game theory and populations dynamics -- see for instance the references \cite{ATHR, AZR, Aguiar_games,  AC98, Rodrigues2}.  The study of  heteroclinic cycles and networks is well-established as an autonomous subject in the dynamical systems community \cite{HS}. 

The present article contributes to the latter by investigating a particular type of \emph{switching dynamics} in a neighbourhood of a heteroclinic network unfolding another  containing  a two-dimensional connection. This configuration often occurs in  $\mathbb{SO}(2)$ and $\mathbb{SO}(3)$--equivariant differential equations \cite[\S 3.3]{AC98}. 
 There are different types of \emph{switching}, leading to increasingly complex behaviour:  
\begin{itemize}
\item \emph{switching at a saddle} (or \emph{switching at a node})  \cite{Aguiar_games} characterised by the existence of initial conditions near an incoming connection to that saddle, whose trajectory follows any of the possible outgoing connections.  An incoming connection does not predetermine the outgoing choice at the saddle.
\medbreak
\item \emph{switching along a heteroclinic connection} \cite{Aguiar_games} which extends the notion of switching at a saddle to initial conditions whose solutions follow a prescribed homo/heteroclinic connection.
\medbreak

\item \emph{infinite switching} \cite{Aguiar_games, ACL05, ALR, HK2010, IR15, Rodrigues2}, which ensures that any infinite sequence of connections in the network is a possible path near the network. This is different from \emph{random switching} in which trajectories shadow the network in a non-controllable  way \cite{PD}. 
\end{itemize}
 \medbreak
 
The absence of \emph{switching along a connection} prevents \emph{infinite switching}.   \\

Kirk and Silber  \cite{KS} studied a network consisting of two cycles and trajectories are allowed to change from a neighbourhood of one cycle to a neighbourhood of the other cycle.  Despite both cycles are numerically observable, there is no sustained switching in this example: a trajectory may switch from one cycle to the other initially but do not switch back again. 
This behaviour is referred as \emph{switching}, although it is a very weak example of this phenomenon. In \cite{CLP},  the expression \emph{railroad switching} is used in relation to switching at a saddle.

Postlethwaite and Dawes  \cite{PD}  found a form of complicated switching  leading to regular and irregular cycling near a network. Castro and Lohse \cite{CL} explored two examples of networks with rich dynamics  without infinite switching. 
Complex behaviour near a network may also arise from the presence of noise-induced switching, but we do not address the presence of noise in the present research article.

 At this moment, it is convenient to distinguish between Lyapunov-\emph{stable} networks from the \emph{unstable} ones. 
 In the first case, the authors of \cite[Th. 2.1]{HK2010} and \cite{Holmes, RodriguesSPM} proved the existence of forward infinite switching near a  homoclinic network: solutions shadow any infinite sequence of connections while approaching the network. In these cases, there is neither backward switching nor suspended horseshoes near the network. 
For stable networks in three-dimensions involving saddle-foci,  the existence of a set of initial conditions with positive Lebesgue measure realising infinite switching  is known. %In \cite{Kirk2010}, the authors   show that switching is ubiquitous in the network, whether or not the network is asymptotically stable.

In \cite{Kirk2010}, the authors have studied   an example in which some solutions  near a heteroclinic network switch   between excursions about  different cycles.  
The authors defined two positive numbers $\delta^{min}$ and  $\delta^{max}$ (\footnote{These two positive numbers depend on the eigenvalues of the linear part of the vector field and the coefficients of the global maps; they are explicitly described in \cite[\S 3]{Kirk2010}.}) such that $\delta^{min}< \delta^{max}$ and:
\begin{itemize}
\item if $\delta^{min}>1$, then  solutions approach the network and spend increasing periods of times near the saddles;
\item if $\delta^{max}<1$, then the network is unstable and almost all trajectories leave the neighbourhood of the network; no estimates are possible. 
\end{itemize}

The ``intermediate'' case defined by $\delta^{min}<1<\delta^{max}$ offers the possibility that some solutions may maintain an average of contraction and expansion, from where switching may emerge.  The mechanism for switching is the presence of a pair of complex eigenvalues in the linearisation of the flow about one of the equilibria in the network. Whether or not an individual trajectory approaches the network or diverges from it depends on the detailed itinerary of that trajectory.  Numerically, most trajectories seem to be attracted to a chaotic or periodic attractor some small distance away from the network.  
The whole network structure are observed in the long term dynamics   even though the network is not stable. As far as we know, there is not a formal proof of the existence of a chaotic attractor  near the network studied in  \cite{Kirk2010}.

In this paper, we concentrate our attention on unstable networks, where forward and backward switching seem to be the consequence of hyperbolic suspended horseshoes in the neighbourhood of the network --  \cite[Th. 2.2]{HK2010}, \cite[Prop. 2]{LR} and \cite{ACL05, ALR,   IR15, Rodrigues2, Rodrigues3}. 
 We assume the initial vector field has sufficient regularity  ($C^2$) in such a way that hyperbolic invariant sets have zero Lebesgue measure. For unstable networks, we know nothing about the Lebesgue measure of the set of initial conditions that realize infinite switching. %Although infinite switching might not be observable in numerics, it is not the case in the present paper.

\medbreak

\textbf{Novelty:} 
In this paper, we describe a one-parameter family of vector fields exhibiting \emph{abundant switching}  near a network, \emph{i.e.} within any small open ball near the network, there exists a set of initial conditions with positive Lebesgue measure shadowing any prescribed infinite path.  
Our  proof relies on the existence of ``large'' strange attractors (in the terminology of \cite{BST98}) near a heteroclinic tangle  which  unfolds an attracting network with a continuum of connections. %We prove that, in any small ball  in the phase space, there is an initial condition realising infinite switching. This phenomenon occurs for a set of positive Lebesgue measure of parameters.
Our main results have been partially motivated by the numerics of \cite{RodLab} and the ideas of \cite[\S 3.2]{Kirk2010} and \cite[\S 5.3.3]{BST98}.  \medbreak

\textbf{Summary:} 
In Table \ref{table2}, we summarise the main results in the literature about infinite switching and their dynamical mechanisms.   We choose to mention only a few for clarity and the choice has been based on our personal preferences. The reader interested in further detail and/or more examples can use the references within those we mention.   We have not included the work \cite{IR15} because their network cannot be reduced to a three-dimensional center manifold, although it contains very complex dynamics involving switching. 

\begin{table}[htb]
\small
\begin{center}
\begin{tabular}{|c|c|c|} \hline 
 {Concept}  & Dynamical  mechanism & References\\
\hline \hline
 
&&\\
&  & \cite{ Holmes, RodriguesSPM} \\
 & Stable: Lyapunov-stable network&  \\  Stable infinite switching & Switching: complex eigenvalues at the nodes   & Theorem 2.1 of \cite{HK2010} \\ &  $\Downarrow$  &   \\ &  Spread solutions in all directions  &  \\ &  while approaching the network &  \\ &&  \\ \hline
\hline 
&&\\
&  &  \\
  & Neither stable nor completely unstable  &  \\  ``Intermediate'' case of & ($\delta^{min}<1<\delta^{max}$ of  \cite{Kirk2010})&Example III   \\  switching && in Section 4 of \cite{Kirk2010}\\&   Average distance from the network &    
  \\ & + complex eigenvalues at one node   & \\& $\Downarrow$  & \\ &  Spread ``some'' solutions in all directions   &  \\ && \\ \hline
\hline 
&&\\
&  &  \\
 & Complex eigenvalues  +& \cite{ACL05, ALR,    LR}  \\ 

  Unstable infinite switching&  2D invariant manifolds intersect transversely &      \\

 &  (when it is possible) &      \\   &  $\Downarrow$  & Theorem 2.2 of \cite{HK2010}    \\ & Suspended hyperbolic horseshoes  &    \\  & $\Downarrow$ &    \\ &Backward and forward switching&\\& &\\ \hline
\hline 
&&\\
&  &  \\
Abundant infinite switching & Suspended horseshoes &   \\  in unstable networks& +  twisting number large enough&   \textcolor{blue}{Novelty}  \\ & $\Downarrow$   &\\ &``Large'' strange attractors \cite{BST98} & \\ &  &\\\hline
  
 \end{tabular}
\end{center}
\label{table2}
\bigskip
\caption{\small Overview of the results in the literature about infinite switching and the contribution of the present article (in blue). The concepts of stable and unstable rely on the Lyapunov stability.}
\end{table} 
\medbreak
\subsection*{Structure of this article}The rest of this paper is organised as follows. 
For reader's convenience,   we have compiled in Section \ref{Definitions} a list of basic  definitions. 
 In Section \ref{s:setting}, we describe precisely our object of study, we review the literature related to it and  we state the main results of the article. The coordinates and other notation  are presented in Section \ref{localdyn}. After reducing the dynamics of the 2-dimensional first return map to the dynamics of a 1-dimensional map in Section \ref{proof Th B},   we collect the main ideas of \cite{Takahasi1} about a precise family of circle maps with singularities in Section \ref{s:results}. These results are refined   in Section \ref{s:unbounded}.  

The proof of Theorems \ref{prop_main} and \ref{thm:C} are performed  in Sections \ref{Prova Th A1} and \ref{Prova Th B1}, respectively.    We perform   illustrative computer experiments using \emph{Matlab} (R2021b, Mathworks, Natick, MA, USA)   for an explicit family of vector fields in Section  \ref{s:example}.  Finally, in Section \ref{s:discussion}, we relate our results with others in the literature, emphazising the role of the \emph{twisting number}.

We have endeavoured to make a self contained exposition bringing together all topics related to the proofs. We revive some useful results from the literature; we hope this saves the reader the trouble of going through the entire length of some referred works to achieve a complete description of the theory. We have drawn illustrative figures to make the paper easily readable.

\section{Preliminaries}\label{Definitions}
In this section, we introduce some terminology for vector fields acting on $\RR^4$ that will be used in the remaining sections.
For $\varepsilon>0$ small enough, consider the one-parameter family of $C^3$--smooth autonomous differential equations
\begin{equation}
\label{general2a}
\dot{x}=g_{\lambda}(x)\qquad x\in \EU^3\subset \RR^4  \qquad  \lambda \in [0, \varepsilon] 
\end{equation}
where $\EU^3$ denotes the three-dimensional unit sphere endowed with the usual topology. Denote by $\varphi_{\lambda}(t,x)$, $t \in \RR$, the  flow associated to \eqref{general2a}. The flow is \emph{complete} (all solutions are defined for all $t\in \RR$) because  $\EU^3$ is a boundaryless compact set. For $n\in \NN$, throughout the article, let us denote by $Leb_n$ the usual Lebesgue measure  of $\RR^n$.

\subsection{Attracting set}
\label{attracting set}

A subset $\Omega$ of $\EU^3$ for which there exists a neighbourhood $\mathcal{U} \subset \EU^3$  satisfying $\varphi_{\lambda}(t,\mathcal{U})\subset \mathcal{U}$ for all $t \in \RR_0^+$ and $$\bigcap_{t\,\in\,\RR^+_0}\,\varphi_{\lambda}(t,\mathcal{U})=\Omega$$ is called an \emph{attracting set} by the flow of \eqref{general2a};  it is not necessarily connected. Its basin of attraction, denoted by $\mathcal{B}(\Omega)$, is the set of points in $\EU^3$ whose orbits have $\omega-$limit in $\Omega$. In this article, we say that $\Omega$ is \emph{asymptotically stable}, or that $\Omega$ is a \emph{global attractor}, if $\mathcal{B}(\Omega)=\EU^3 $.

\subsection{Heteroclinic structures}
\label{ss:Bylov_cycle}
Suppose that $P_1$ and $P_2$ are two hyperbolic saddle-foci of \eqref{general2a}. There is a {\em heteroclinic cycle} associated to $P_1$ and $P_2$ if
$$W^{u}(P_1)\cap W^{s}(P_2)\neq \emptyset \qquad \text{and} \qquad W^{u}(P_2)\cap W^{s}(P_1)\neq \emptyset.$$ For $i, j \in \{1,2\}$, the non-empty intersection of $W^{u}(P_i)$ with $W^{s}(P_j)$ is called a \emph{heteroclinic connection} from $P_i$ to $P_j$, and will be denoted by $[P_i \rightarrow  P_j]$; this set may be either a single trajectory or a union of trajectories (continua of connections \cite{AC98}). 
A \emph{heteroclinic network} is a connected union of heteroclinic cycles.
Although heteroclinic cycles involving equilibria are not a generic feature within differential equations, they may be structurally stable within families of systems which are equivariant under the action of a compact Lie group  due to the existence of flow-invariant subspaces.

\subsection{Lyapunov exponents}
\label{def: LE}
A \emph{Lyapunov exponent} may be seen as an average exponential rate of divergence or convergence of nearby trajectories in the phase space. 
 Let $M$ be a compact, connected and smooth Riemannian two-dimensional manifold and $f:M\rightarrow M$  a diffeomorphism. 
 By the Oseledets' Theorem, for Lebesgue almost points $x\in{M}$, there is a splitting
$$T_{x}M=E^{1}_{x}\oplus {E^{2}_{x}},$$ (called the \emph{Oseledets' splitting}) and real numbers $\lambda_{1}(x) \geq \lambda_{2}(x)$
(called \emph{Lyapunov exponents}) such that $Df({x})(E^{1}_{x})=E^{1}_{f(x)}$ and  $Df({x})(E^{2}_{x})=E^{2}_{f(x)}$  and
$$\underset{n\rightarrow{\pm\infty}}{\lim}\,\frac{1}{n}\log\|Df^{n}({x})(v^j)\|=\lambda_j(f,x)$$
for any $v^{j}\in{E^{j}_{x}\backslash\{\vec{0}\}}$,  $j=1,2$, where  $\|\cdot\|$ denotes the euclidean norm in $ M$. 
For $x\in M$, if either $\lambda_1(f,x)>0$ or $\lambda_2(f,x)>0$ , then one has exponential divergence of nearby orbits.
In this case, we say that there exists an orbit with a \emph{positive Lyapunov exponent}.
Its presence  implies that trajectories whose initial conditions are hard to be distinguished  behave  quite differently in the future.

\subsection{``Large'' strange attractor}
\label{def: large}
 Based on \cite{BST98}, we define ``large'' strange attractor. For $\varepsilon>0$, a   \emph{large strange attractor} of a two-dimensional dissipative diffeomorphism defined on an annulus parametrized by $[0,2\pi]\times [0, \varepsilon]$ (\footnote{The first component is the angular coordinate where $0$ is identified with $2\pi$ and the second component is the height component. This set is also called by \emph{circloid}.}), is a compact invariant set $\Lambda$ with the following properties:\\
\begin{enumerate}
\item the basin of attraction of $\Lambda$   contains a non-empty open set (and thus has positive Lebesgue measure);
%\medbreak
\item there is a dense orbit in $\Lambda$ with a positive Lyapunov exponent ($\Leftrightarrow$ exponential growth of the derivative along its orbit);
%\medbreak
\item  the strange attractor winds around the whole annulus $[0,2\pi]\times [0, \varepsilon]$.\\
%\medbreak
\end{enumerate}
A vector field possesses a ``large'' strange attractor if the first return map to an annular cross-section does.

 \subsection{Infinite (forward) switching}
Let $\Gamma$ be a heteroclinic network associated to $\{P_1,\ldots, P_n\}$, a set of $n$ invariant saddles, where $n\in \NN$.

\begin{definition}
\label{Def2}
If $k\in\NN$, a \emph{finite path of order $k$} on $\Gamma$ is a sequence
     $
     \{\gamma_{1},\ldots,\gamma_{k}\}
     $
     of one-dimensional heteroclinic connections in $\Gamma$ such that $\omega(\gamma_i)=\alpha(\gamma_{i+1})$, for all $i\in \{1, ..., k-1\}$. We use the notation $\sigma^{k}=\{\gamma_1, ...,\gamma_k \}$ for this type of finite path. For an \emph{infinite path}, take $k\in \NN$ and denote it by $\sigma^\infty$.
     \end{definition}

Let $N_{\Gamma}$ be a neighbourhood of the network $\Gamma$ (with a finite number of heteroclinic connections) and let $W_j\subset N_{\Gamma}$ be a neighbourhood of  $P_j$, $j \in\{1, \ldots, n\}$. For each heteroclinic connection $\gamma_i$ in $\Gamma$, consider a point $p_i\in \gamma_i$ and a neighbourhood $V_i\subset N_{\Gamma}$ of $p_i$. The collection of these neighbourhoods should be pairwise disjoint -- see Figure \ref{switching1}.

\begin{definition}
Given neighbourhoods as above, for $k\in \NN$, we say that the trajectory of a point $q\in \EU^3$ \emph{follows the finite path of order} $k$, say
$\sigma^k$, if there exist two monotonically increasing sequences of times $(t_{j})_{j\in \{1,\ldots,k+1\}}$ and $(z_{j})_{j\in \{1,\ldots,k\}}$ such that for all $j \in \{1,\ldots,k\}$, we have $t_{j}<z_{j}<t_{j+1}$ and:\\

\begin{enumerate}
\item[(i)]
$\varphi_\lambda (t,q)\subset N_{\Gamma}$ for all $t\in \, \, [ t_{1},t_{k+1}]$;
\item[(ii)]
$\varphi_\lambda (t_{j},q) \in W_j$ for all $j \in \{1,\ldots,k+1\}$ and $\varphi_\lambda (z_{j},q)\in V_{j}$ for all $j \in \{1,\ldots,k\}$;
\item[(iii)] for all $j=1,\ldots,k$ there exists a proper subinterval $I\subset \, \, (z_{j},z_{j+1})$ such that, given $t\in \, \, (z_{j},z_{j+1})$, $\varphi_\lambda(t,q)\in W_{j+1}$ if and only if $t\in I$.\\
\end{enumerate}
\end{definition}
The notion of a trajectory following an infinite path can be stated similarly. Along the paper, when we refer to points that follow a path, we mean that their trajectories do it.  Based on  \cite{ALR, Rodrigues2}, we define:

\begin{definition} There is:\\
\label{Def1}
\begin{enumerate}
\item[(i)]  \emph{finite switching} of order $k$ near $\Gamma$ if  for each finite path of order $k$, say $\sigma^k$, and for each neighbourhood $N_{\Gamma}$, there is a trajectory in $N_{\Gamma}$ that follows $\sigma^k$ and
\item[(ii)]  \emph{infinite forward switching} (or simply \emph{switching}) near $\Gamma$  by requiring that for each infinite path and for each neighbourhood $N_{\Gamma}$, there is a trajectory in $N_{\Gamma}$ that follows it.
\end{enumerate}
\end{definition}

 In general, switching is defined for positive time; this is why it is called by \emph{forward switching}. We may define analogously \emph{backward switching} by reversing the direction of the variable $t\in \RR$.
 
An infinite path on $\Gamma$ can be considered as a pseudo-orbit of (\ref{general2a}) with infinitely many discontinuities. Switching near $\Gamma$ means that any pseudo-orbit in $\Gamma$ can be realized. In \cite{HK2010}, using \emph{connectivity matrices}, the authors gave an equivalent definition of switching, emphasising the possibility of coding all trajectories that remain in a given neighbourhood of the network in both finite and infinite times.

 \begin{figure}[h]
\begin{center}
\includegraphics[height=8.5cm]{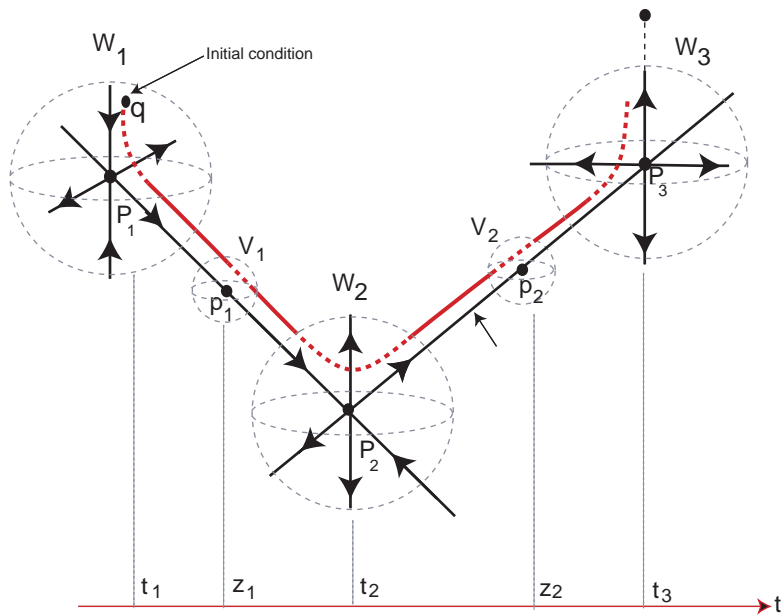}
\end{center}
\caption{\small  The trajectory associated to $q\in \EU^3$ \emph{follows the finite path of order} $3$.}
\label{switching1}
\end{figure}

 \begin{definition}
 \label{def:abundant}
 We say that system (\ref{general2a}) exhibits \emph{abundant infinite switching} near $\Gamma$ if for every neighbourhood of $\Gamma$ and every infinite path, the set of initial conditions within the neighbourhood that follows the path has positive Lebesgue measure. \end{definition}

\section{Setting and Main results}
\label{s:setting}
 In this section, we enumerate the main assumptions concerning the configuration of our  network and we state the main results of the article.
\subsection{Starting point}
\label{ss:oc}
For $\varepsilon>0$ small enough and $r\geq 3$, consider the one-parameter family of $C^r$--smooth differential equations (endowed with the usual $C^r$--topology):
\begin{equation}
\label{general2.1}
\dot{x}=g_{\lambda}(x)\qquad x\in \EU^3 \qquad  \lambda \in [0, \varepsilon]
\end{equation}
 satisfying the following hypotheses for $\lambda=0$:

\bigbreak
\begin{enumerate}
 \item[\textbf{(H1)}] \label{B1}  There are two different equilibria, say $P_1$ and $P_2$.
 \bigbreak
 \item[\textbf{(H2)}] \label{B2} The eigenvalues  of $Dg_0(X)$ are:
 \medbreak
 \begin{enumerate}
 \item[\textbf{(H2a)}] $E_1$ and $ -C_1\pm \omega_1 i $ where $C_1>E_1>0, \quad \omega_1>0$, \quad for $X=P_1$;
 \medbreak
 \item[\textbf{(H2b)}] $-C_2$ and $ E_2\pm \omega_2 i $ where $C_2> E_2>0, \quad \omega_2>0$,  \quad for $X=P_2$.
 \end{enumerate}
 \end{enumerate}

 \bigbreak
The equilibrium   $P_1$ has a $2$-dimen\-sional stable and $1$-dimen\-sional unstable manifold that will be denoted by $W^s(P_1)$ and $W^u(P_1)$, respectively.  In a similar way,   $P_2$ has a 1-dimen\-sional stable and $2$-dimen\-sional unstable manifold and the terminology is  $W^s(P_2)$ and $W^u(P_2)$.   For $M\subset \EU^3$, denoting by $\overline{M}$ the topological closure of $M$, we also assume that:

 \begin{enumerate}
  \bigbreak
  \item[\textbf{(H3)}]\label{B3} The manifolds $\overline{W^u(P_2)}$ and $\overline{ W^s(P_1)}$ coincide and   $\overline{W^u(P_2)\cap W^s(P_1)}$ consists of a two-sphere (continuum of connections). \bigbreak
\end{enumerate}

   \begin{enumerate}
\item[\textbf{(H4)}]\label{B4} There are two trajectories, say  $\gamma_1, \gamma_2 \subset W^u(P_1)\cap W^s(P_2)$, one in each connected component of $\EU^3\backslash \overline{W^u(P_2)}$ (one-dimensional connections).
\end{enumerate}
 \bigbreak

For $\lambda=0$, the two equilibria $P_1$ and $P_2$, the two-dimensional heteroclinic connection from $P_2$ to $P_1$ referred in \textbf{(H3)} and the two trajectories listed in  \textbf{(H4)}  build a \emph{heteroclinic network} which we denote by $\Gamma$. This network  has an \emph{attracting} character and the dynamics nearby is well known:

  \begin{lemma}[\cite{LR, LR2016}]
\label{attractor_lemma}
The network $\Gamma$ is a global attractor and does not exhibit switching.
\end{lemma}

Since $\Gamma$ is a global attractor, we may find an open neighbourhood $\mathcal{U}$ around $\Gamma$ having its boundary transverse to the flow   and such that every solution starting in $\mathcal{U}$ remains in it and is forward asymptotic to $\Gamma$. This (small) neighbourhood will be called the \emph{absorbing domain of $\Gamma$}. 

Let $W_1$
and $W_2$ be small disjoint neighbourhoods of $P_1$ and $P_2$ with disjoint boundaries $\partial W_1$ and $\partial W_2$, respectively. Trajectories starting at $\partial W_1$ near $W^s(P_1)$ go into the interior of $W_1$ in positive time, then follow the connection from $P_1$ to $P_2$, go inside $W_2$, and then come out at $\partial W_2$. Let $\mathcal{Q}$ be a piece of trajectory like this from $\partial W_1$ to $\partial W_2$.
Now join its starting point to its end point by a line segment, forming a closed curve, that we call the  \emph{loop} of $\mathcal{Q}$.
The loop of $\mathcal{Q}$ and   $\Gamma$ are disjoint closed sets. 
Following \cite{LR2015}, we say that the two saddle-foci $P_1$ and $P_2$ in $\Gamma$ have the same \emph{chirality} if the loop of every trajectory is linked to $\Gamma$ in the sense that the two closed sets cannot be disconnected by an isotopy.  
  From now on, we assume the following technical condition:
\medbreak
\begin{enumerate}
\item[\textbf{(H5)}] \label{B5} The equilibria $P_1$ and $P_2$ have the same chirality.
\end{enumerate}
\medbreak

\subsection{Perturbing term}
The parameter $\lambda$ acts on the dynamics of \eqref{general2.1} in the following way:  
 \medbreak
\begin{enumerate}
\item[\textbf{(H6)}] \label{B5} For $ \lambda > 0$,  the two trajectories $\gamma_1$ and $\gamma_2$ within $W^u(P_1)\cap W^s(P_2)$ persist.
\end{enumerate}

 \medbreak
\begin{enumerate}
\item[\textbf{(H7)}]\label{B7} For $\lambda > 0$, the  manifolds $W^u(P_2)$ and $W^s(P_1)$  intersect transversely.
\end{enumerate}

 \medbreak

 and
\medbreak

\begin{enumerate}
\item[\textbf{(H8)}] \label{B8} There exist $\varepsilon>0$ and $\lambda_1>0$ for which the global maps associated to  the connections $$[P_1 \rightarrow  P_2] \qquad \text{and}\qquad[P_2 \rightarrow  P_1]$$ are given, in local coordinates, by the \emph{Identity map} and by the expression:
$$
\left(\begin{array}{c} x  \\ y\end{array}\right) \mapsto \left(\begin{array}{c} \xi  \\ 0 \end{array}\right)+ \left(\begin{array}{cc} 1 & 0  \\ 0 & 1 \end{array}\right) \left(\begin{array}{c} x  \\ y\end{array}\right) +\lambda \left(\begin{array}{c} \Phi_1(x,y)  \\ \Phi_2(x,y)\end{array}\right)  \qquad \text{for} \qquad \lambda \in [0, \lambda_1]
$$
respectively, where $\xi\in \RR,$ $$\Phi_1:\EU^1 \times [-\varepsilon, \varepsilon] \rightarrow \RR, \qquad \Phi_2: \EU^1 \times [-\varepsilon, \varepsilon] \rightarrow \RR$$ are $C^2$--maps and $  \Phi_2(x, 0)$ is a Morse function with two nondegenerate critical points (say $\pi/2$ and $3\pi/2$) and  two zeros (say $0$ and $2\pi$). This assumption will be clearer  later in Section \ref{localdyn}.
\end{enumerate}
\bigbreak

\begin{remark}
 When $\lambda$ varies, the invariant manifolds associated to $P_1$ and $P_2$  vary smoothly with $\lambda$. For the sake of simplicity, we omit this dependence in the notation.
 \end{remark}

 For  $r \geq 3$, denote by  $\mathfrak{X}_{Byk}^r(\EU^3)$, the family of $C^r$--vector fields on $\EU^3$ endowed with the $C^r$--Whitney topology, satisfying Properties \textbf{(H1)--(H8)}. Note that for $\lambda>0$, the flow of $g_\lambda$ has  a Bykov network whose dynamics have been explored in \cite{Bykov00, LR, Rodrigues3}, which explains the subscript \emph{Byk} in $\mathfrak{X}_{Byk}^r(\EU^3)$.

\bigbreak

\subsection{Constants}
\label{constants1}
We settle the following notation on the \emph{saddle-values} of $P_1$, $P_2$ and $\Gamma$:
\begin{equation}
\label{constants}
\delta_1 = \frac{C_1}{E_1 }>1, \qquad \delta_2 = \frac{C_2}{E_2 }>1, \qquad \delta=\delta_1\, \delta_2>1
 \end{equation}
and on the \emph{twisting number} defined as:
\begin{equation}
\label{constants2}
  K_\omega= \frac{E_2 \, \omega_1+C_1\, \omega_2 }{E_1E_2}>0.
 \end{equation}

The terminology ``\emph{twisting number}'' is due to its effect  on the dynamics: if it is large enough, then it forces the spread of trajectories around the two-dimensional manifold $W^u(P_2)$.

\subsection{Main results}\label{main results}

Let $N_\Gamma$ be a neighborhood of the attractor $\Gamma$ that exists for $\lambda=0$.  For $\lambda_0>0$ small enough and $r\geq 3$, let $\left(g_{\lambda}\right)_{\lambda \in \, (0, \lambda_0] }$ be a one-parameter family of vector fields in $\mathfrak{X}_{Byk}^r(\EU^3)$, where $\lambda_0 \in (0, \lambda_1)$ (see the meaning of $\lambda_1$ in \textbf{(H8)}).
\bigbreak
\begin{proposition}\label{thm:0}
Let $g_{\lambda} \in\mathfrak{X}_{Byk}^r(\EU^3)$, $\lambda \in [0, \lambda_0]$.
Then, there is $\tilde\varepsilon>0$  such that the first return map to a given cross section $\Sigma$ to $\Gamma$ may be written (in local coordinates $(x,y)$ of $\Sigma$)   by:
{ $$
\mathcal{G}_{\lambda}(x,y)=\left[ x+\xi+\lambda \Phi_1(x,y) - K_\omega  \ln |y+\lambda\Phi_2(x,y)| \pmod{2\pi}, \, \, \, (y+\lambda\Phi_2(x,y))^\delta \right]+\dots
$$}where $\xi\in \RR$, $$(x,y)\in \mathcal{D}=\{x\in \RR \pmod{2\pi}, \quad y/\tilde \varepsilon  \in [-1, 1] \quad \text{and} \quad y+\lambda\Phi_2(x,y) \neq 0\}\subset \Sigma$$ and the ellipses stand for asymptotically small  terms depending on $x$ and $y$ converging  to zero (as $y$ goes to zero) along with their first derivatives.

\end{proposition}
\bigbreak

The proof of Proposition \ref{thm:0} is adapted from \cite{Rodrigues_2022_DCDS} and will be revisited in Section \ref{proof Th A}.
Since $\delta>1$, for $\lambda>0$ small enough, the second component of $\mathcal{G}_{\lambda}$ is contracting and its dynamics is dominated by the \emph{family of circle maps} with singularities:
$$
h_a(x)= x + a +\xi - K_\omega \ln |\Phi_2(x)|$$
where:\\
\begin{itemize}
\item $x\in \EU^1\equiv \RR / (2\pi \ZZ)$;\\
\item $\xi\in \RR$;\\
\item $ a  = - K_\omega \ln \lambda \pmod{2\pi}$, $\lambda \in \, (0, \lambda_0]$ and\\
\item $\Phi_2(x)\equiv \Phi_2(x,0)$ is the map defined in  \textbf{(H8)}. The singularities of $h_a$ are the zeros of~$\Phi_2$.\\
\end{itemize}

 \medbreak
 The next result shows that, for any small unfolding of $g_0$, in the $C^3$--Whitney topology, there is a sufficiently large twisting number  prompting  the persistence of ``large'' strange attractors (for $\mathcal{G}_\lambda$ defined in Proposition \ref{thm:0}).
 \medbreak
  \begin{maintheorem}
 \label{prop_main}
Let $\left(g_{\lambda}\right)_{\lambda \in \, (0, \lambda_0] } \in\mathfrak{X}_{Byk}^r(\EU^3)$ with $r\geq 3$.
For  $K_\omega > 0$ large enough,     there exists a set $ \Delta_\lambda\subset [0, \lambda_0]$ with positive Lebesgue measure such that if $\lambda \in \Delta_\lambda$, then the flow of $g_\lambda$  contains a ``large'' strange attractor.
\end{maintheorem}
 \medbreak
%If $\lambda \in \Delta_\lambda$, the non-wandering dynamics of $g_\lambda$ is governed by the two-dimensional set $W^u(P_2)$.
The proof of Theorem \ref{prop_main} is performed  in Section \ref{Prova Th A1} by reducing the dynamics of the two-dimensional first return map $\mathcal{G}_\lambda$  to the dynamics of a one-dimensional map $h_a$.  
Since $\Phi_2$ has zeros, the map  $h_a$ in \eqref{h_a} (called later by \emph{singular limit}) is not defined in a compact set and  has singularities with unbounded derivatives near them. The classical theory of Rank-one strange attractors developed by Wang and Young \cite{WY} \textbf{cannot} be applied directly to the case under consideration.

\bigbreak

  For $\lambda>0$, the network $\Gamma$ is broken and a more complex network  emerges as a consequence of \textbf{(H6)} and \textbf{(H7)}.  This network consists of the two connections of $\textbf{(H6)}$ and infinitely many connections resulting from  the non-empty transverse intersection of  $W^u(P_2)$ and $W^s(P_1)$ (Proposition 4 of \cite{Rodrigues3}). 
  For $\lambda>0$,  let $\Gamma_\lambda$ be the network associated to $P_1$ and $P_2$ with a finite number of heteroclinic connections from $P_2$ to $P_1$ (this number might be arbitrarily large).
  The authors of \cite {LR} proved that there is \emph{infinite switching near $\Gamma_\lambda$}. In this paper, we go further by proving the existence of abundant infinite switching in the sense of Definition~\ref{def:abundant}. This is the content of the next result:

\begin{maintheorem}\label{thm:C}
Let $(g_{\lambda})_\lambda\in\mathfrak{X}_{Byk}^r(\EU^3)$ with $r\geq 3$.
For  $K_\omega > 0$ large enough,     there exists a set $\Delta_\lambda\subset [0, \lambda_0]$ with positive Lebesgue measure such that if $\lambda \in \Delta_\lambda$, then  the network $\Gamma_\lambda$  exhibits  abundant infinite switching. This phenomenon is realized by any non-empty ball of initial conditions lying in  $\mathcal{U}$.
     \end{maintheorem}

 Theorem \ref{thm:C}  may be seen as a consequence of Theorem \ref{prop_main} and its proof is done in Section  Section \ref{Prova Th B1}.

\bigbreak

\subsection{Remarks on the hypotheses}
 \label{s:digestive}
 In this subsection, we point out some remarks about the Hypotheses \textbf{(H1)--(H8)} and the main results.

\begin{remark}
The ``large'' strange attractors of Theorem \ref{prop_main} contain non-uniformly hyperbolic horseshoes. 
 The horseshoes whose existence has been proven in \cite{LR, LR2016, Rodrigues3}, when restricted to a compact set, are uniformly hyperbolic but they do not correspond to the whole  non-wandering set associated to $\mathcal{U}$; they are restricted to a small ``window'' near $\Gamma_\lambda, \lambda>0$.   

\end{remark}

\begin{remark}
The full description of the bifurcations associated to $\Gamma$ is a phenomenon of codimension three.  Nevertheless,  the setting described by \textbf{(H1)--(H8)}  is natural in $\mathbb{SO}(2)$--symmetry-breaking contexts \cite{LR, RodLab} and also in the scope of some unfoldings of the Hopf-zero singularity \cite{BIS}.
\end{remark}

 \begin{remark}
 Hypothesis \textbf{(H6)} corresponds to the \emph{partial symmetry-breaking} considered in Section 2.4 of \cite{LR}.
The setting described by \textbf{(H1)--(H8)} generalizes Cases (2) and (3) of \cite{Rodrigues2019}.
 Hypothesis \textbf{(H8)} is generic if we consider one of the simplest scenarios for the splitting of a two-dimensional sphere defined by the coincidence of the two-dimensional invariant manifolds.
  \end{remark}

\begin{remark}  

For $\lambda>0$, the flow of $g_\lambda$ exhibits a \emph{heteroclinic tangle}.
 The distance between $W^u_{\loc} (P_2)$ and $W^s_{\loc} (P_1)$  in a cylindrical  cross-section to $\Gamma$ may be computed using the  \emph{Melnikov integral} \cite[Appendix A]{RodLab}; the map $\Phi_2(x)$ may be seen as the Melnikov integral, up to a possible reparametrisation. The proofs  of Theorems  \ref{prop_main} and \ref{thm:C} are analogous if $ \Phi_2(x, 0)$ is a Morse function with a finite number of nondegenerate critical points and  a finite number of zeros.

  \end{remark}

\begin{remark}
The analytical  expressions for the transitions maps along the heteroclinic connections $[P_1 \rightarrow  P_2]$ and $ [P_2 \rightarrow  P_1]$ could be written as a general  \emph{Linear map}  as:
$$
\left(\begin{array}{c} x  \\ y\end{array}\right) \mapsto \left(\begin{array}{cc} a  &0\\ 0&\frac{1}{a}\end{array}\right)\left(\begin{array}{c} x  \\ y\end{array}\right) ,
$$
and by
$$
\left(\begin{array}{c} x  \\ y\end{array}\right) \mapsto \left(\begin{array}{c} \xi_1  \\ \xi_2 \end{array}\right)+ \left(\begin{array}{cc} b_{1} & b_2  \\ c_1 & c_2 \end{array}\right) \left(\begin{array}{c} x  \\ y\end{array}\right) +\lambda \left(\begin{array}{c} \Phi_1(x,y)  \\ \Phi_2(x,y)\end{array}\right)
$$
respectively, where $a\geq 1$, $\xi_1, \xi_2, b_1, b_2, c_1, c_2 \in \RR$. For the sake of simplicity, we restrict to the case $a= b_1=c_2=1$ and $\xi_1=\xi_2= b_2=c_1=0$. This simplifies the computations and is not a restriction \cite[\S 6]{Rodrigues3}.
\end{remark}

\section{Local and transition maps}\label{localdyn}

In this section we will analyze the dynamics near the network $\Gamma_\lambda$, $\lambda \geq 0$ through local maps, after selecting suitable coordinates in the neighbourhoods of the saddle-foci $P_1$ and $P_2$.  Note that $\Gamma_0\equiv \Gamma$.

\subsection{Local coordinates}
\label{ss:lc}
We use the local coordinates near the equilibria $P_1$ and $P_2$ introduced in  \cite{LR2016}.
\medbreak

We  consider cylindrical neighbourhoods  $W_1$ and $W_2$  in ${\RR}^3$ of $P_1 $ and $P_2$, respectively, of radius $\rho=\varepsilon>0$ and height $z=2\varepsilon$.
After a linear rescaling of the variables, we  assume that  $\varepsilon=1$.
Their boundaries consist of three components: the cylinder wall parametrised by $x\in \RR\pmod{2\pi}$ and $|y|\leq 1$ with the usual cover $$ (x,y)\mapsto (1 ,x,y)=(\rho ,\theta ,z),$$ and two discs, the top and bottom of the cylinder. We take polar coverings of these disks $$(r,\varphi )\mapsto (r,\varphi , \pm 1)=(\rho ,\theta ,z)$$
where $0\leq r\leq 1$ and $\varphi \in \RR\pmod{2\pi}$.
The local stable manifold of $P_1$, $W^s(P_1)$, corresponds to the circle parametrised by $ y=0$. In $W_1$ we use the following terminology:\\
\begin{itemize}
\item
$\In(P_1)$, the cylinder wall of $W_1$,  consisting of points that go inside $W_1$ in positive time;\\
\item
$\Out(P_1)$, the top and bottom of $W_1$,  consisting of points that go outside $W_1$ in positive time. It has two connected components.
\end{itemize}
We denote by $\In^+(P_1)$ the upper part of the cylinder, parametrised by $(x,y)$, $y\in\, \, (0,1)$ and by $\In^-(P_1)$ its lower part parametrized by $(x,y)$, $y\in (-1,0)$.

\medbreak
The cross-sections obtained for the linearisation of $g_\lambda$ around $P_2$ are dual to these. The set $W^s(P_2)$ is the $z$-axis intersecting the top and bottom of the cylinder $W_2$ at the origin of its coordinates. The set
$W^u(P_2)$ is parametrised by $z=0$, and we use:\\

\begin{itemize}
\item
$\In(P_2)$, the top and bottom of $W_2$,  consisting of points that go inside $W_2$ in positive time;\\
\item
$\Out(P_2)$,  the cylinder wall  of $W_2$,  consisting of points that go inside $W_2$ in negative time, with $\Out^+(P_2)$ denoting its upper part, parametrised by $(x,y)$, $y\in\, \, (0,1)$ and $\Out^-(P_2)$  its lower part parametrised by $(x,y)$, $y\in\,  (-1,0)$. \\
\end{itemize}

We will denote by $W^u_{\loc}(P_2)$ the portion of $W^u(P_2)$  that goes from $P_2$ to $\In(P_1)$ not intersecting the interior of $W_1$ and by $W^s_{\loc}(P_1)$  the portion of $W^s(P_1)$ outside $W_2$ that goes directly  from $\Out(P_2)$ into $P_1$, as shown in Figure \ref{transition_new}. The flow is transverse to these cross-sections and the boundaries of $W_1$ and of $W_2$ may be written as   $\overline{\In(P_1) \cup \Out (P_1)}$ and  $
\overline{\In(P_2) \cup \Out (P_2)}$. The orientation of the angular coordinate near $P_2$ is chosen to be compatible with the direction induced by Hypotheses \textbf{(H4)} and  \textbf{(H5)}.

 \bigbreak

\subsection{Local maps near the saddle-foci}
Adapting \cite{Deng1} (see also Proposition 3.1 of \cite{DIK}), the trajectory of  a point $(x,y)$ with $y>0$ in $\In^+(P_1)$ leaves $W_1$ at
 $\Out(P_1)$ at
\begin{equation}
\mathcal{L}_{1 }(x,y)=\left(y^{\delta_1} + S_1(x,y;  \lambda),-\frac{\omega_1 \, \ln y}{E_1}+x+S_2(x,y;  \lambda) \right)=(r,\varphi),
\label{local_v}
\end{equation}
where $S_1$ and $S_2$ are smooth functions which depend on $\lambda$ and satisfy:
\begin{equation}
\label{diff_res}
\left| \frac{\partial^{k+l+m}}{\partial x^k \partial y^l  \partial \lambda ^m } S_i(x, y; \lambda)
\right| \leq C y^{\delta_1 + \sigma - l},
\end{equation}
and $C$ and $\sigma$ are positive constants and $k, l, m$ are non-negative integers. Similarly, a point $(r,\varphi)$ in $\In(P_2) \backslash W^s_{\loc}(P_2)$ leaves $W_2$ at $\Out(P_2)$ at
\begin{equation}
\mathcal{L}_2(r,\varphi )=\left(-\frac{\omega_2\, \ln r}{E_2}+\varphi+ R_1(r,\varphi ;  \lambda),r^{\delta_2 }+R_2(r,\varphi;  \lambda )\right)=(x,y)
 \label{local_w}
\end{equation}
where $R_1$ and $R_2$ satisfy a  condition similar  to (\ref{diff_res}). The terms $S_1$, $S_2$,  $R_1$, $R_2$ correspond to asymptotically small terms that vanish when the  components $y$ and $r$ go to zero. \bigbreak

\subsection{The global map}\label{transitions}
The coordinates on $W_1$ and $W_2$ are chosen so that $[P_1\rightarrow P_2]$ connects points with $z>0$ (resp. $z<0$) in $W_1$ to points with $z>0$  (resp. $z<0$) in $W_2$. Points in $\Out(P_1)  $ near $W^u(P_1)$ are mapped into $\In(P_2)$ along a flow-box around each of the connections $[P_1\rightarrow P_2]$. We will assume that the transitions
$$\Psi_{1 \rightarrow  2}^+\colon \quad \Out^+(P_1) \quad \rightarrow  \quad \In^+(P_2)$$
$$\Psi_{1 \rightarrow  2}^-\colon \quad \Out^-(P_1) \quad \rightarrow  \quad \In^-(P_2)$$
do not depend on  $\lambda$ and may be considered as the \emph{Identity map}, as a consequence of Hypothesis \textbf{(H4)} and  \textbf{(H8)}. Denote by $\eta^+, \eta^-$ the following maps
$$\eta^+ =\mathcal{L}_{2} \circ \Psi_{1 \rightarrow  2}^+ \circ \mathcal{L}_{1 }\colon \quad \In^+(P_1) \quad \rightarrow  \quad \Out^+(P_2)$$
$$\eta^- =\mathcal{L}_{2} \circ \Psi_{1 \rightarrow  2}^- \circ \mathcal{L}_{1 }\colon \quad \In^-(P_1) \quad \rightarrow  \quad \Out^-(P_2).$$
From \eqref{local_v} and \eqref{local_w}, omitting high order terms in $y$ and $r$,  we conclude that, in local coordinates of $\In (P_1)\backslash W^s_\loc(P_1)$ ($\Leftrightarrow$ $|y|\neq 0$),  we have:
\begin{equation}\label{eqeta}
\eta^\pm (x,y)= \left(x-K_\omega \log |y| \,\,\,\pmod{2\pi}, \,y^{\delta} \right)   
\end{equation}
with
\begin{equation}\label{delta e K}
\delta=\delta_1 \delta_2>1 \qquad \text{and} \qquad  K_\omega= \frac{C_1\, \omega_2+E_2\, \omega_1}{E_1 \, E_2} > 0.
\end{equation}

 \begin{figure}[h]
\begin{center}
\includegraphics[height=5.5cm]{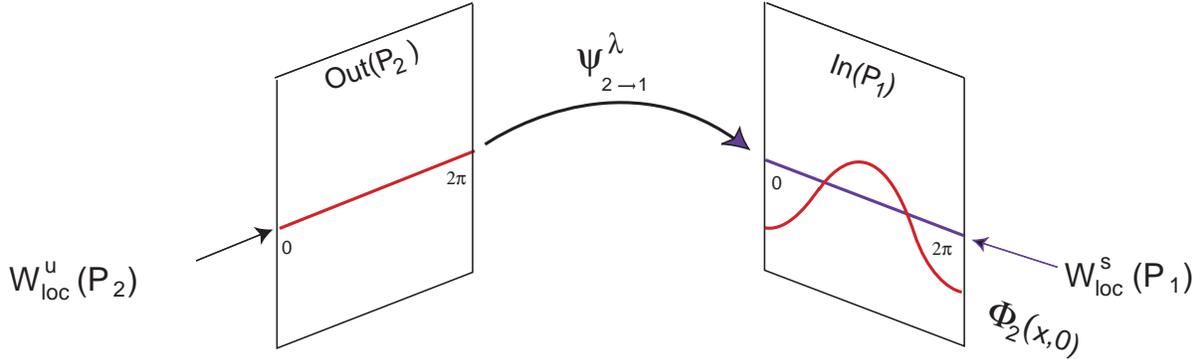}
\end{center}
\caption{\small  Illustration of the transition map  $\Psi_{2 \rightarrow  1}^\lambda $ from $\Out(P_2)$ to $\In (P_1)$.  The graph of $\Phi_2(x,0)\equiv \Phi_2(x)$ may be seen as the first hit of $W^u_\loc(P_2)$ to $\In (P_1)$.}
\label{transition_new}
\end{figure}

 Using  \textbf{(H7)} and \textbf{(H8)}, for $ \lambda \in \,  [0, \lambda_1]$, we   have a well defined transition map
$$\Psi_{2 \rightarrow  1}^{\lambda}:\Out(P_2)\rightarrow  \In(P_1)$$
that depends on the parameter $\lambda$, given by:
\begin{equation}\label{transition21}
\Psi_{2 \rightarrow  1}^{\lambda}(x,y)=\left(\xi +x +\lambda \Phi_1(x,y), \,y +\lambda \Phi_2(x,y) \right).
\end{equation}

 To simplify the notation, in what follows we will sometimes drop the superscript $ \lambda$, unless there is some  risk of misunderstanding.
By \textbf{(H8)}, the map $\ln |\Phi_2(x,0)| $ has two singularities and two critical points -- see Figure \ref{transition_new}.

\section{Proof of Proposition \ref{thm:0}}
\label{proof Th A}
The proof of Proposition  \ref{thm:0} is straightforward by considering $\Sigma=\Out(P_2)$ and by composing the local and global maps constructed in Section \ref{localdyn}.  More specifically, for
$\lambda\in [0,\lambda_0]$, with $\lambda_0 \in [0,\lambda_1]$, let:

\begin{equation}\label{first return 1}
\mathcal{G}_{\lambda} = \eta \circ  \Psi_{2 \rightarrow  1}^{\lambda}:  \quad  \mathcal{D} \subset \Out(P_2)\backslash W^s_{\loc}(P_1) \quad \rightarrow  \quad \mathcal{D} \subset \Out(P_2)  
\end{equation}
be the first return map to $\mathcal{D}$, where $\mathcal{D}\subset  \Out(P_2)$ is the set of initial conditions $(x,y) \in \Out(P_2)$ whose solution returns to $\Out(P_2)$. Up to high order terms, composing     (\ref{eqeta}) with   (\ref{transition21}), the analytic expression of $\mathcal{G}_{\lambda}$ is given by:
\label{first1}
{ \begin{eqnarray*}
\mathcal{G}_{\lambda}(x,y)&=& \left[  x+\xi+\lambda \Phi_1(x,y) - K_\omega  \ln |y+\lambda\Phi_2(x,y)| \pmod{2\pi}, \, \, \left(y +\lambda \Phi_2(x,y) \right)^\delta\right]\\
&=& \left(\mathcal{G}^1_{\lambda}(x,y), \mathcal{G}^2_{\lambda}(x,y)\right).
\end{eqnarray*}}

Initial conditions $(x,y)$ that do not return to $\Out(P_2)$ are contained in $W^s(P_1)$;  such points are parametrized by $y+\lambda\Phi_2(x,y)=0$.
Although the map is $C^\infty$ (where it is well defined), the approximation of $\mathcal{G}_\lambda$ may be performed in a $C^2$--topology since the local maps $\mathcal{L}_1$ and $\mathcal{L}_2$   may be taken to be $C^{r-1}$ (observe that $r\geq 3$ is the class of differentiability of the initial vector field) and the global maps are assumed to be $C^2$--embeddings.  

 \begin{remark}
 When $\lambda = 0$, we may write (for $|y|\neq 0$):
 $$
 \mathcal{G}_{0}(x,y)= \left( \xi+ x - K_\omega  \ln |y| \pmod{2\pi}, \, \, y^\delta\right). $$
 This means that the $y$-component is contracting and thus the dynamics is governed by the $x$-component. This is consistent with the fact that   $\Gamma$ is attracting (Lemma \ref{attractor_lemma}).  If $y=0$ and $x\in \EU^1$, then $(x,y)\in W^s(P_1)$, implying that the associated trajectory does not return to $\Sigma=\Out(P_2)$.  \end{remark}

\section{Singular limit}
\label{proof Th B}
 
 In this section, we compute the singular limit set associated to $\mathcal{G}_\lambda$ defined in Proposition \ref{thm:0}. The formal definition of \emph{singular limit} may be found in \cite{Rodrigues_2022_DCDS, WY}.
 
\subsection{Change of coordinates}
\label{change_of_coordinates}
For $\lambda \in\,\,  (0, \lambda_0) $ fixed and $(x,y) \in \Out(P_2)$, let us make the following change of coordinates:
  \begin{equation}
  \label{change1}
  \overline{x} \mapsto {x} \qquad \text{and} \qquad \overline{y} \mapsto \frac{y}{\lambda} (\Leftrightarrow y=\overline{y}\lambda).
  \end{equation}
Taking into account that:
\begin{eqnarray*}
\mathcal{G}^1_{\lambda} (x,y) &=& x+\xi+{\lambda} \Phi_1(x,y) - K_\omega  \ln |y+{\lambda}\Phi_2(x,y)| \pmod{2\pi}\\
&=&   x+\xi+{\lambda} \Phi_1(x,y) - K_\omega  \ln \left|{\lambda} \left(\frac{y}{\lambda }+\Phi_2(x,y)\right)\right| \pmod{2\pi}\\
 &=&   x+\xi+{\lambda}  \Phi_1(x,y) - K_\omega  \ln \lambda -K_\omega \ln  \left| \frac{y}{\lambda}+\Phi_2(x,y) \right| \pmod{2\pi}\\
  \mathcal{G}^2_{\lambda} (x,y) &=& (y+\lambda \Phi_2(x,y))^\delta =  \lambda^\delta \left(\frac{y}{\lambda}+ \Phi_2(x,y)\right)^\delta,\\
  \end{eqnarray*}
  we may write:\\
\begin{eqnarray*}
\mathcal{G}^1_{\lambda} (x,\lambda \overline{y}) &=& x+\xi+{\lambda}  \Phi_1(x,\lambda\overline{y}) - K_\omega  \ln \lambda -K_\omega \ln  \left|  \overline{y}+\Phi_2(x,\lambda\overline{y})\right| \pmod{2\pi}\\ \\
\mathcal{G}^2_{\lambda} (x,\lambda \overline{y}) &=&  \lambda^{\delta-1} \left(\overline{y}+ \Phi_2(x,\lambda \overline{y})\right)^\delta. \\
  \end{eqnarray*}

\subsection{Reduction to a singular limit}
\label{ss: reduction}
In this subsection, we compute the singular limit of $\mathcal{G}_\lambda$ written in the coordinates $(x,\overline{y})$ of Subsection \ref{change_of_coordinates}, for $\lambda \in\,\, (0, \lambda_0)$.
Let  $k: \RR^+ \rightarrow \RR$ be the invertible map defined by $$k(x)= -K_\omega \ln (x),$$
whose graph is depicted in Figure \ref{scheme3A}.
 Define now the decreasing sequence $(\lambda_n)_n$ such that, for all $n\in \NN$, we have:\\
\begin{enumerate}
\item  $\lambda_n\in\, (0, \lambda_0)$ (the meaning of $\lambda_0$ comes from Proposition \ref{thm:0}) and \\
\item $k(\lambda_n) \equiv 0 \pmod{2\pi}$.\\
\end{enumerate}
\medbreak
 \begin{figure}[h]
\begin{center}
\includegraphics[height=10cm]{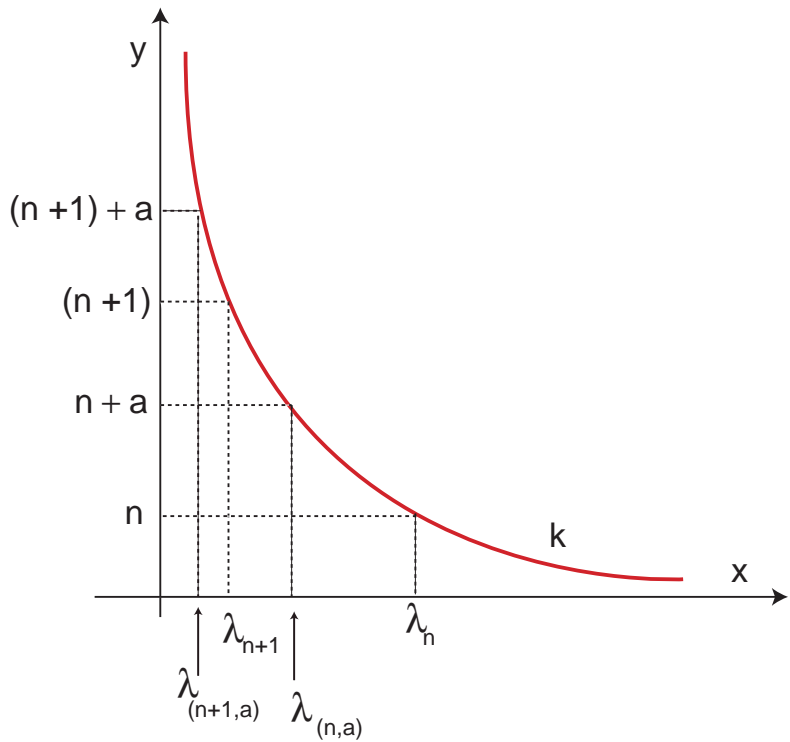}
\end{center}
\caption{\small  Graph of $k(x)= -K_\omega \ln (x)$ and  illustration of the sequences $(\lambda_n)_n$ and $(\lambda_{(n, a)})_n$ for a fixed $a \in [0,2\pi)$.}
\label{scheme3A}
\end{figure}

Since $k$ is   invertible,  for $a \in \EU^1 \equiv [0, 2\pi)$ fixed and $n\geq n_0\in \NN$, define:
\begin{equation}
\label{sequence1}
\lambda_{(a, n)}= k^{-1}\left[ k(\lambda_n)+a\right]\,\,  \in \,\, (0, \lambda_0),
\end{equation}
as shown in Figure  \ref{scheme3A}. We may write:
\begin{equation}
\label{sequence2}
k\left(\lambda_{(a, n)}\right)= -K_\omega \ln (\lambda_{n})+a=a \pmod{2\pi}.
\end{equation}
\bigbreak

For $a\in \EU^1$, the following lemma establishes  the convergence of the map $\mathcal{G}_{\lambda_{(n,a)}}$  to a two-dimensional map as $n \rightarrow +\infty$, ($\|.\|_{\textbf{C}^r}$ represents the norm   in the $C^r$--topology for $r\geq2$):

\begin{lemma}
\label{important lemma}
The following equality holds:
$$
\lim_{n\rightarrow +\infty} \|\mathcal{G}_{\lambda_{(n,a)}} (x,\overline{y}) -( h_a(x,\overline{y}), \textbf{0})\|_{\textbf{C}^2} =0$$
where $\textbf{0}$ is the null map and
\begin{equation}
\label{circle map}
h_a(x, \overline{y})= x+a-K_\omega \ln|\overline{y} +\Phi_2(x,\overline{y})| +\xi.
\end{equation}
\end{lemma}

\begin{proof}

Using \eqref{sequence2}, we have
{ \begin{eqnarray*}
\mathcal{G}^1_{\lambda_{(n,a)}} (x,\overline{y})&=&x+\xi+{\lambda_{(n,a)}}  \Phi_1(x,\overline{y})- K_\omega  \ln \lambda_{(n,a)} -K_\omega \ln   |\overline{y}+\Phi_2(x,\overline{y})|  \pmod{2\pi} \\
&=&x+\xi+{\lambda_{(n,a)}}  \Phi_1(x,\overline{y})+a  -K_\omega \ln  |\overline{y}+\Phi_2(x,\overline{y})| \pmod{2\pi}\\
  \mathcal{G}^2_{\lambda_{(n,a)}} (x,\overline{y}) &=&  {\lambda_{(n,a)}}^{\delta-1} \left(\overline{y}+ \Phi_2(x,\overline{y})\right)^\delta.\\
  \end{eqnarray*}}
    Since $\dpt \lim_{n\rightarrow +\infty} {\lambda_{(n,a)}}=0$ we may write:
  \begin{eqnarray*}
\lim_{n\rightarrow +\infty}   \mathcal{G}^1_{\lambda_{(n,a)}} (x,\overline{y}) &=& x+\xi+a  -K_\omega \ln  |\Phi_2(x,0)|  \pmod{2\pi}\\
  \lim_{n\rightarrow +\infty}   \mathcal{G}^2_{\lambda_{(n,a)}} (x,\overline{y}) &=& 0\\
  \end{eqnarray*}
 and we get the result.
\end{proof}

Denoting  $\Phi_2(x,0)$ by $\Phi_2(x)$, the map 
$$h_a(x) =x+\xi+a  -K_\omega \ln  |\Phi_2(x)| $$ is the \emph{singular limit} in the spirit of \cite{WY}; it has two nondegenerate critical points  and two singularities -- see Figure \ref{singularities1}. The map $h_a$ is not defined on a compact set and its derivative explodes to $\infty$ near the singularities. Since $h_a$ is not a \emph{Misiurewicz-type} map in the sense of \cite{WY}, the Theory of Rank-one attractors cannot be applied to this case and results should be adapted.

 \begin{figure}[h]
\begin{center}
\includegraphics[height=15cm]{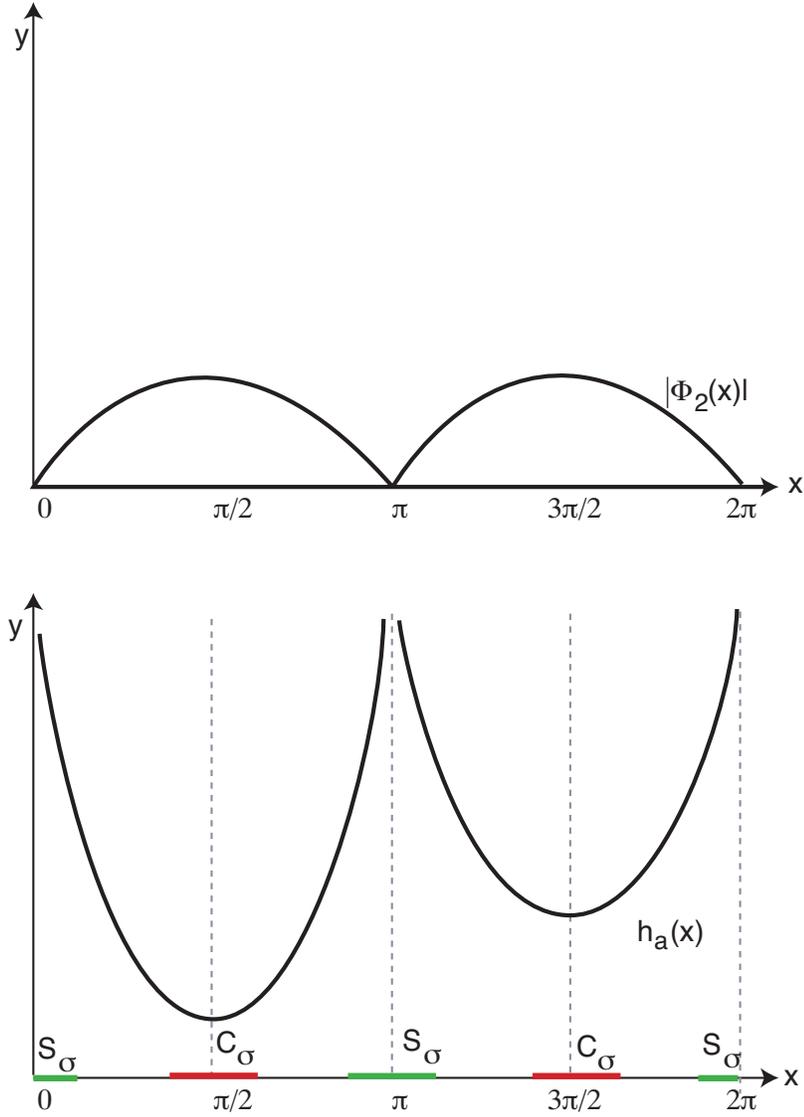}
\end{center}
\caption{\small  Graph of $|\Phi_2(x)|$ (upper image) and $h_a(x)=x+\xi+a-K_\omega \ln |\Phi_2(x)|$ (lower image). For $\sigma>0$, the sets $\mathcal{C}_\sigma$ and  $\mathcal{S}_\sigma$  represent the neighbourhoods of the critical and singular sets of $h_a$.}
\label{singularities1}
\end{figure}

\section{Related results}
\label{s:results}
To make a coherent presentation, in this section we collect the ideas from  \cite{Takahasi1} that will be used in the sequel.  We hope this saves the reader the trouble of going though the entire lengh of \cite{Takahasi1}.
From now on, we locate \emph{the singular and critical} sets of 
\begin{equation}
\label{h_a}
h_a(x)=x+\xi+a-K_\omega \ln |\Phi_2(x)|
\end{equation} where $x\in \EU^1\backslash\{0, \pi\}$. These sets are endowed with the distance $dist$ (euclidean metric on $\EU^1\equiv\RR / (2\pi \ZZ)$):
\begin{eqnarray*}
\mathcal{S}=\{x\in \EU^1: \Phi_2(x)=0\}=\{0, \pi\}\ & \mapsto & \text{Singular set}\\
 \mathcal{C}=\{x\in \EU^1: \Phi_2'(x)=0\}=\left\{\frac{\pi}{2}, \frac{3\pi}{2}\right\}& \mapsto & \text{Critical set}\\
\end{eqnarray*}
 and, for $\sigma>0$, define (see the red and green bold lines in Figure \ref{singularities1}):
$$
\mathcal{S}_\sigma= \{x\in \EU^1: dist(x, \mathcal{S})\leq \sigma\}\quad \text{and} \quad \mathcal{C}_\sigma= \{x\in \EU^1: dist(x, \mathcal{C})\leq \sigma\}.$$

Since $$h'_a(x)= 1-K_\omega\frac{\Phi'_2(x)}{\Phi_2(x)} \quad \text{and} \quad h''_a(x)= K_\omega \frac{\Phi''_2(x)\Phi_2(x)-(\Phi'_2(x))^2}{(\Phi'_2(x))^2},$$ it follows that:
\begin{lemma}[Lemma 2.1 of \cite{Takahasi1}, adapted]
\label{aux1}
There exist $K_0>1$ and $\varepsilon>0$ such that the following inequalities hold for $K_\omega$ sufficiently large:
\begin{enumerate}
\item for all $x\in \EU^1\backslash\{0, \pi\}$, we have $$
\frac{K_\omega}{K_0} \frac{dist(x,\mathcal{C})}{dist(x,\mathcal{S})}\leq |h_a'(x)| \leq {K_\omega}{K_0}\frac{dist(x,\mathcal{C})}{dist(x,\mathcal{S})};$$
\item if $x\in \EU^1$, we have:
\begin{eqnarray*}
|h_a'(x)| \geq \frac{K_\omega}{K_0} \varepsilon &\text{if} & x\notin \mathcal{C}_\varepsilon\\\\
\frac{K_\omega}{K_0} <|h_a''(x)|<{K_\omega}{K_0}&  \text{if} & x\in \mathcal{C}_\varepsilon.
\end{eqnarray*}
 
\end{enumerate}
\end{lemma}

Item (1) of Lemma \ref{aux1} says that the derivative of $h_a$ at $x\in \EU^1\backslash\{0, \pi\}$ goes like the inverse of the distance of $x\in \EU^1\backslash \mathcal{S}$ to the singular set $\mathcal{S}$: if $x$ approaches $  \mathcal{S}$, then $h_a'(x)$ explodes.  This does happen for \emph{Misiurewicz-type} maps (see \S 5 of \cite{Rodrigues_2022_DCDS}). We are interested in the dynamics of an interval very close to  the critical point $c\in \mathcal{C}$ whose dynamics  ``generates'' an irreducible strange attractor. In order to do that, we need to introduce some terminology from \cite{Takahasi1}. 
If $c\in \mathcal{C}$ is a critical point of $\Phi_2$, we set:\\
\begin{eqnarray*}
c_n&=&h_a^{n+1}(c), \quad n \in \NN_0 \qquad\text{(orbit of the critical point under } h_a)\\ \\
\xi &=& \frac{1}{(K_\omega)^{ 1/6}}>0  \\ \\
\xi(n) &=& \frac{1}{(K_\omega)^{ n/10^6}}>0 , \quad n \in \NN_0 \\ \\
J(x)&=&|h_a'(x)|\\ \\
J^n(x)&=& J(x).J(h_a(x)).J(h^2_a(x))...J\left(h_a^{n-1}(x)\right)\\ \\
d_n(c_0) &=& \frac{dist(c_n,\mathcal{C}) .dist(c_n,  \mathcal{S})}{J^n(c_0)}  , \quad n\in \NN\\ \\
D_n(c_0) &=& \frac{1}{\sqrt{K_\omega}} \left[   \sum_{i=0}^{n-1}\frac{1}{d_i(c_0)}\right]^{-1}>0 \\\\
 I_n(c)&=& h_a^{-1}[c_0+D_{n-1}(c_0), c_0+D_n(c_0)], \qquad n\geq 2\\
 %I_{-p}&\mapsto &\text{mirror image of $I_p(c)$ with respect to $c$.}\\ \\
\end{eqnarray*}
For $c\in \mathcal{C}$, $\xi >0$  and $n\in \NN$, denote also the following sets: \\
\begin{eqnarray*}
  \Delta_n&=&\{a\in \EU^1: (h_a^{i+1}(\mathcal{C}))\cap (\mathcal{C}_\xi\cup \mathcal{S}_\xi)=\emptyset, \qquad  \text{for all} \quad i \in \{1, ..., n\}\}\\ \\
\Delta&=&\left\{a\in \EU^1:  |h_a^n(h_a(c))|\geq (K_\omega)^\frac{n}{10^3}, \qquad  \text{for all} \quad  n\in \NN\right\}.\\
 \end{eqnarray*}

\subsection*{Interpretation}
In what follows we point out some comments about the above constants, sets and intervals:\\
\begin{itemize}
\item For all $n\in \NN$, $a\in \Delta$ and $c\in \mathcal{C}$, we have $D_{n-1}(h_a(c))>D_n(h_a(c))$.\\
\item For all $a\in \Delta$  and $n\in \NN$, the  set $I_n(c)$ is the interval $$\left(c+\sqrt{\frac{D_n(h_a(c))}{K_0 K_\omega}}; c+\sqrt{\frac{D_{n-1}(h_a(c))}{K_0 K_\omega}}\right)$$ whose amplitude tends to $0$ as $K_\omega$ goes to $+\infty$.\\
\item If $x\in   I_{n}(c)$, then $|h_a(x)-h_a(c)|\leq D_{n-1}(h_a(c))$.  The derivatives along the orbit of $h_a(x)$ shadow that of the orbit of $h_a(c)$ for $n-1$ iterates -- see  \cite[Lemma 2.2]{Takahasi1}. \\  %(page 537 of \cite{Takahasi1}). \\
\item If $a\in \Delta$, then for all $n\in \NN$, $J^n(c_0)\neq 0$, $d_n(c_0)>0$ and $D_n(c_0)>0$.\\
\item The set $\Delta$ depends on $K_\omega$ (by construction).\\
\end{itemize}
 
\bigbreak

 The goal of the next result is twofold. The first item  estimates the length of  the set of parameters $a$ for which we have $h_a^{i+1}(\mathcal{C})\cap (\mathcal{C}_\xi\cup \mathcal{S}_\xi)=\emptyset$  for all $i \in \{1, ..., N\}$; the second means that, if $a\in \Delta$ then  there exists an exponential growth of the derivative of $h_a$ along the orbit of the critical point $c\in \mathcal{C}$ (this is the main Theorem of \cite{Takahasi1}).\\

\begin{lemma}[\cite{Takahasi1}, adapted] The following conditions hold: \\
\label{lemma7.2}
\begin{enumerate}
 \item There exist $N_1\in \NN$  and  $K_\omega^\star>0$ such that if $N>N_1$ and $K_\omega>K_\omega^\star$ then: $$Leb_1(\Delta_{N})\geq 2\pi -\frac{1}{(K_\omega)^\frac{1}{9}}.$$\\
\item   $|(h_a^n)'(h_a(c))|\geq (K_\omega)^{n/1000}$ for all $n\in \NN$, $a\in \Delta$ and $c\in \mathcal{C}$.\\
\end{enumerate}
\end{lemma}
\bigbreak

 Using (1) of Lemma \ref{lemma7.2}, we may conclude that $\dpt \lim_{K_\omega \rightarrow +\infty}  Leb_1(\Delta_N)=2\pi$. Indeed, following the calculations of page 549 of \cite{Takahasi1}, adapted to our purposes, we get:

\begin{eqnarray*}
2\pi \geq Leb_1(\Delta)&=&Leb_1(\Delta_N)-\sum_{n=N+1}^\infty Leb_1(\Delta_{n-1}\backslash \Delta_n) \\
&\overset{\text{Lemma   \ref{lemma7.2}}}{\geq}& \left[2\pi -(K_\omega)^{-\frac{1}{9}}\right] - \sum_{n=N+1}^\infty  \left[(K_\omega)^{-\frac{n}{10^{11}}}- (K_\omega)^{-\frac{n}{3.10^6}}\right]\\
 \end{eqnarray*}
Since $$\dpt \lim_{K_\omega \rightarrow +\infty}  \left[2\pi -(K_\omega)^{-\frac{1}{9}}\right] -\dpt \lim_{K_\omega \rightarrow +\infty} \sum_{n=N+1}^\infty  \left[(K_\omega)^{-\frac{n}{10^{11}}}- (K_\omega)^{-\frac{n}{3.10^6}}\right]=2\pi,$$ then 
  $\dpt \lim_{K_\omega \rightarrow +\infty}  Leb_1(\Delta)=2\pi$.
 
\bigbreak

The next result ensures an exponential growth of derivatives outside $\mathcal{C}_\xi$. Expansion is lost due to returns to this set. Let $N_1\gg 1$ and $ K_\omega^\star\gg 1$ be as in Item (1) of Lemma \ref{lemma7.2}.
\bigbreak
\begin{lemma}[\cite{Takahasi1}, adapted]
\label{lemma7.3}
For $N>N_1$ and $K_\omega>K_\omega^\star$ ,   the following conditions are valid for $a\in \Delta_{N}$: \\
\begin{enumerate}
\item  for $n\geq 1$, if $x, h_a(x), h_a^2(x), ..., h_a^{n-1}(x)\notin \mathcal{C}_{\xi(N_1)}$, then $J^n(x) \geq \xi(N_1)  (K_\omega)^\frac{2n}{1000}$; \\
\item  for $n\geq 1$, if $f^n(x)\in \mathcal{C}_{\xi(N_1)}$, then $J^n(x)\geq (K_\omega)^\frac{2n}{1000}$;\\
\item $ \forall i\in \{1, ..., N_1\}, \quad dist(c_i, \mathcal{C}), dist(c_i, \mathcal{S})> \frac{1}{(K_\omega)^6}.$

\end{enumerate}
\end{lemma}

 \bigbreak

\bigbreak

 As pointed out on page 536  of \cite{Takahasi1}, $N_1$ and $K_\omega^\star$ of Lemmas \ref{lemma7.2} and \ref{lemma7.3} are independent and may be taken as $K_\omega^\star\gg N_1\gg 1$.
 The main result of \cite{Takahasi1} concludes about the existence of a strange attractor for $h_a$ ``generated'' by the critical point $c$ (see  Item (2) of  Lemma \ref{lemma7.2}). The breakthrough of the present article is to prove that this strange attractor is ``large'' in the sense of Subsection \ref{def: large}. This is why we need to refine and extend their results, which is the goal of next section.

 \begin{figure}[h]
\begin{center}
\includegraphics[height=15cm]{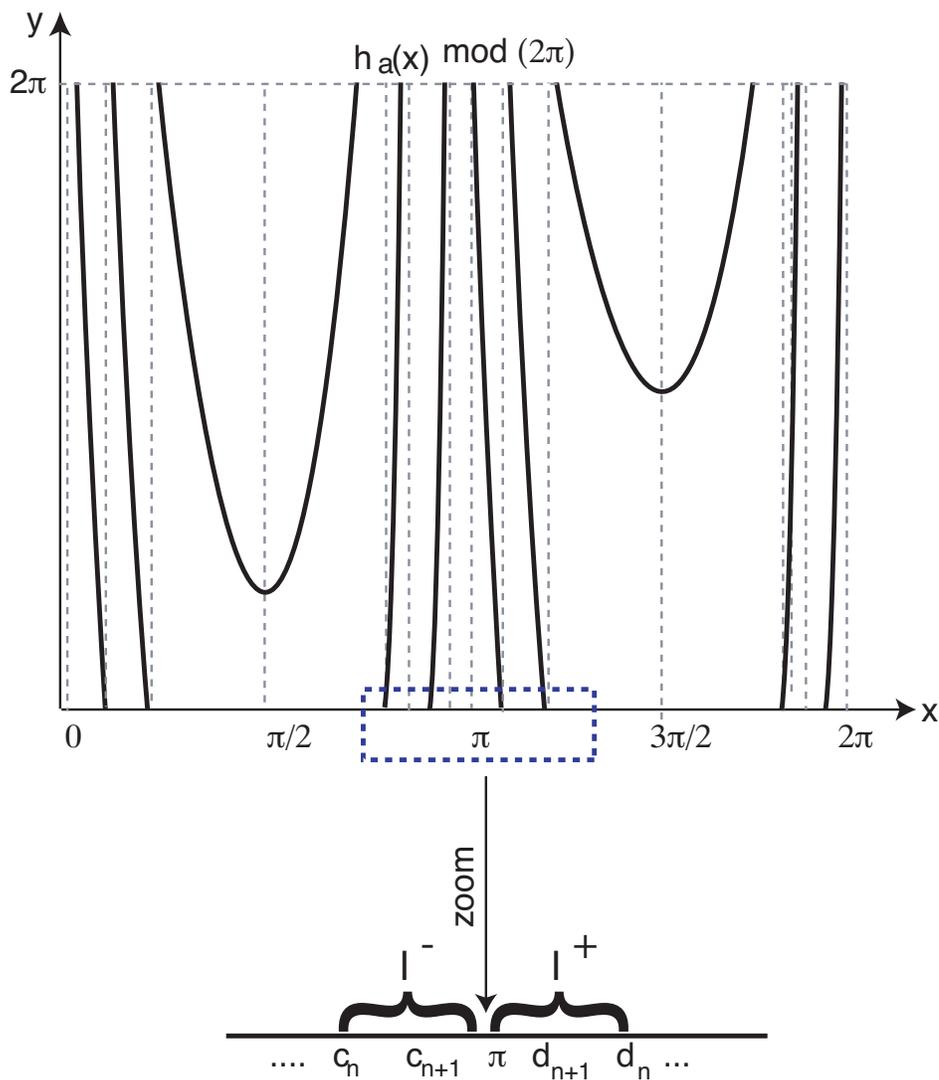}
\end{center}
\caption{\small  Graph of   $h_a\pmod{2\pi}$ and illustration of the sequences $(c_n)_n$ and $(d_n)_n$ of Lemma \ref{cover S1}.}
\label{singularities2}
\end{figure}

\section{Effects of the unbounded derivative}
\label{s:unbounded}
The goal of this section is to refine the results of Section \ref{s:results} adapted to $h_a$.
The following lemma says that there are small intervals near $s\in \mathcal{S}$ where the image of $h_a$ covers the whole circle $\EU^1$. Elements of $\mathcal{S}$ blow up the derivatives of $h_a$, allowing expansion and enforcing shift dynamics ($\Leftrightarrow$ chaos). 
 
\begin{lemma}
\label{cover S1}
For $s\in \mathcal{S}$ and $a\in \EU^1$, there exists a nested sequence of intervals of the type $(c_n, d_n)\in [0, 2\pi]$ such that $h_a$ is injective in $(c_n, c_{n+1})$ and $(d_n, d_{n+1})$ and 
$$
h_a((c_n,c_{n+1}])= h_a([d_{n+1},d_{n}))=\EU^1. 
$$
\end{lemma}
\begin{proof}
We suggest the reader to follow the proof by observing Figure \ref{singularities2}.
For $s=\pi$ (if $s=0$, the proof is similar) and $a\in \EU^1$, since
$$
\lim_{x\rightarrow s^+}|\Phi_2(x)|=\lim_{x\rightarrow s^-}|\Phi_2(x)|=0
$$
it follows that
$$
\lim_{x\rightarrow s^+}h_a(x)=\lim_{x\rightarrow s^-} h_a(x)=+\infty.
$$
As depicted in Figure \ref{singularities2}, define the intervals $\mathcal{I}^-$ and $\mathcal{I}^+$ subintervals of $[\pi/2, \pi)$ and $(\pi, 3\pi/2]$, where $|\Phi_2|$ is monotonically  increasing and decreasing, respectively. They exist because the equation $h'_a(x)=0$ has a unique solution within one of the previous intervals (as a consequence of \textbf{(P8)}).  Indeed,
\begin{eqnarray*}
 h_a'(x)=0 &\Leftrightarrow& 1-K_\omega\frac{\Phi_2'(x)}{\Phi_2(x)}=0 \\
 &\Leftrightarrow& K_\omega\frac{\Phi_2'(x)}{\Phi_2(x)}=1.
   \end{eqnarray*}

Define the sequences $c_n<\pi<d_n$ as:\\
\begin{itemize}
\item $h_a(c_n)=h_a(d_n)=2n\pi$, where  $n\in \NN $\\
\item for all $n\in \NN$, $c_n\in \mathcal{I}^-$ and $d_n\in \mathcal{I}^+$. \\
\end{itemize}
Now it is easy to check that 
$$
h_a((c_n,c_{n+1}])= h_a([d_{n+1},d_{n}))=[0,2\pi] \equiv \EU^1.
$$ 
\end{proof}

\begin{lemma}
\label{union1}
Let $a\in \Delta$ and let $\varepsilon>0$ be arbitrarily small. For any non-degenetate interval $I\subset [0,2\pi) \backslash \mathcal{S}$ of lengh $\varepsilon>0$, there exists a subinterval $I_1\subset I$ and $N_2 \in \NN$ such that  $N_2>N_1$, and $h_a^{N_2}(I_1)$ coincides with one of the components of $\mathcal{C}_\varepsilon \cup \mathcal{S}_\varepsilon$.
 
\end{lemma}
\begin{proof}
Let us fix $\varepsilon>0$ and let $\tilde{N}_2>N_1$   such   that $\xi=\xi(\tilde{N}_2)<\varepsilon$.
Let us iterate  the interval $I$ by $h_a$, deleting all parts that fall into $\mathcal{C}_\xi \cup \mathcal{S}_\xi$. 
Suppose that this may be continued
up to step $n\in \NN$.%, and that for every $ i \leq n$  none of these deleted segments is less than $2\xi$ in terms of length. 

Under the conditions of Lemma \ref{lemma7.3}, the dynamics of $h_a$ is uniformly expanding outside  the set $\mathcal{C}_\xi$. 
Therefore, the number of deleted segments at step $  i \leq n$ is $\leq 2$ (otherwise the result is proved). Then, as depicted in Figure \ref{singularities4}, the Lebesgue measure of the deleted parts in $I$ is less or equal than
 $$
 4\xi \left(1+(K_\omega)^{-2.10^{-3}}+(K_\omega)^{-2.2.10^{-3}}+\ldots +(K_\omega)^{-2.n10^{-3}}\right),
 $$
 and  the Lebesgue measure of the undeleted segment  in $h_a^n(I)$ is greater or equal than
 $$
\left( \varepsilon -   4\xi \left(1+(K_\omega)^{-2.10^{-3}}+(K_\omega)^{-2.2.10^{-3}}+\ldots +(K_\omega)^{-2.n10^{-3}}\right)\right) .\xi .(K_\omega)^{2.n.10^{-3}}
 $$
 which is greater than $2\pi$ after a finite number of iterates, say $N_2$. The interval $I_1$ is one of the  connected components of the pre-image of     $\mathcal{C}_\xi \cup \mathcal{S}_\xi$ under the map $h_a^{N_2}(I)$.
 This proves the lemma. 
  
\end{proof}

  \begin{figure}[h]
\begin{center}
\includegraphics[height=6.0cm]{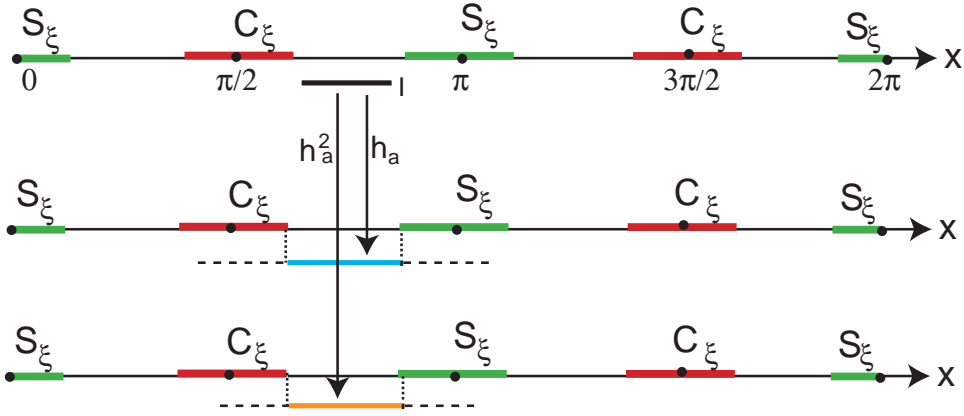}%\\\includegraphics[height=10.0cm]{Numerics2}
\end{center}
\caption{\small  We iterate  the interval $I$ by $h_a$, deleting all parts that fall into $\mathcal{C}_\xi \cup \mathcal{S}_\xi$. For two iterations, the Lebesgue measure of the deleted parts in $I$ is less or equal than
 $ 4\xi \left(1+(K_\omega)^{-2.10^{-3}}+(K_\omega)^{-2.2.10^{-3}}\right).$ } 
\label{singularities4}
\end{figure}

The previous result says that after a finite number of iterates  of $h_a$, any (non-degenerate) interval covers one of the components of $\mathcal{C}_\xi \cup \mathcal{S}_\xi$.
Therefore, we have two disjoint cases: \\
\begin{description}
\item[Case A] $h^{N_2}_a(I_1)$ coincides with one component of $ \mathcal{S}_\xi$. By Lemma \ref{cover S1}, the interval $I_1$ is  sent by $h^{N_2+1}_a(I_1)$  into the whole $\EU^1$.\\
\item[Case B] $h^{N_2}_a(I_1)$ coincides with one component of $ \mathcal{C}_\xi$. Therefore part of the interval follows the orbit of the critical point and $h^{N_2}_a(I_1)\supset I_{N_3}(c)$ for some $N_3>N_2$.\\
\end{description}

For the sake of completeness, we present the following elementary result that will be used in the sequel.

 \begin{lemma}
 \label{lemma8.3}
 For $a>1$ and $N \in \NN$, the following equality holds:
 $$
 \sum_{i=1}^{N}\frac{1}{a^i}= \frac{a^{N}-1}{a^{N}(a-1)}<\frac{1}{a-1}.
$$
 \end{lemma}
 
 \begin{proof} The proof of this result is quite elementary taking into account the formula for the sum of $N$ terms of a geometric sum. For the sake of completeness, we present the proof without deep details:
$$
  \dpt \sum_{i=1}^{N}\frac{1}{a^i}=  \frac{1}{a} \left(\frac{1-\frac{1}{a^{N}}}{1-\frac{1}{a}}\right) = \frac{a^{N}-1}{a^{N}(a-1)}  < \frac{a^{N}}{a^{N}(a-1)} =\frac{1}{(a-1)}.
$$
 \end{proof}

In the following result, let $N_1\in \NN$ be as  in Lemma \ref{lemma7.2}.
 \begin{lemma}
 \label{tec_lemma}
The following inequality holds for  all $i \in \{1, ..., N\}$ where $N>N_1$:
$$J^{N}(c_0) \geq J^i(c_0)\left(\frac{(K_\omega)^{\frac{5}{6}}}{K_0}\right)^{N-i}$$.
 \end{lemma}
 
 \begin{proof}
 The proof follows from the chain of inequalities:
 \begin{eqnarray*}
 J^{N}(c_0) &\overset{\text{Definition}}=& J(c_0).J(h_a(c_0)).J(h_a^2(c_0))\ldots J(h_a^{N-1}(c_0))\\ \\
 &\overset{\text{Lemma \ref{aux1}}}\geq & J^{N-1}(c_0). \frac{K_\omega}{K_0} \frac{dist(c_0,\mathcal{C})}{dist(c_0,\mathcal{S})} \\ \\
 &\overset{\text{Item (3) of Lemma \ref{lemma7.3}}}\geq & J^{N-1}(c_0). \frac{K_\omega}{K_0} \frac{(K_\omega)^{-1/6}}{dist(c_0,\mathcal{S})} \\
  &\overset{0<dist(c_0,\mathcal{S})<1}\geq & J^{N-1}(c_0).  \frac{(K_\omega)^{\frac{5}{6}}}{K_0} \\
  &\geq & J^{N-2}(c_0). \left(\frac{(K_\omega)^{\frac{5}{6}}}{K_0}\right)^2 \\
    && (\ldots) \\
  &\geq & J^{i}(c_0). \left(\frac{(K_\omega)^{\frac{5}{6}}}{K_0}\right)^{N-i}\\
  \end{eqnarray*}
 \end{proof}

 The following lemma says that after a finite number of iterations of $h_a$, the set $h^{N_3}_a(I_1)$ covers the whole circle $\EU^1$.
It concludes the series of results extending the approximations of Section \ref{s:results}. We claim the existence of positive real numbers $k_2, k_3$ and $k_4$  (which depend on $N_3$) without their explicit expression. However, they are clear if we go deeper into the proof. 
 \bigbreak
 \bigbreak
  
\begin{lemma} 
\label{lema auxiliar}
Under the terminology of Section \ref{s:results}, the following inequalities hold:\\
\begin{enumerate}
\item There exists $k_2>0$ such that $\dpt\sum_{i=0}^{N_3-1}  \frac{J^i(c_0)}{dist(c_i,\mathcal{C})dist(c_i,\mathcal{S})}\leq \frac{k_2}{(K_\omega)^{5/6}-1} $.\\\\
\item There exists $k_3>0$ such that if $K_\omega$ is large then ${J^{N_3}(c_0)D_{N_3}(c_0)}\geq  k_3 (K_\omega)^{1/3}$.\\\\
\item  If $K_\omega$ is large enough, then $h_a^{N_3+1} (I_{N_3}(c))=\EU^1$.\\\\
 
\end{enumerate}
\end{lemma}

\begin{proof}
The proof follows from the previous results. We proceed to explain in detail all the steps.
\begin{enumerate}

\item 
\begin{eqnarray*}
\sum_{i=0}^{N_3-1}  \frac{J^i(c_0)}{dist(c_i,\mathcal{C})dist(c_i,\mathcal{S})} &\overset{\text{Lemma \ref{tec_lemma} and } N_3>N_1}{\leq}&\sum_{i=0}^{N_3-1}  \frac{J^{N_3}(c_0)(K_\omega)^{\frac{5(i-N_3)}{6}}K^\star}{k_1^2}\quad \text{where}\quad K^\star>0\\
&\leq& \frac{J^{N_3}(c_0)K^\star}{k_1^2}\sum_{i=0}^{N_3-1}  \frac{1}{(K_\omega)^{\frac{5(N_3-i)}{6}}} \\
&\overset{\text{Lemma \ref{lemma8.3}}}{\leq}& \frac{J^{N_3}(c_0)K^\star}{k_1^2} \frac{1}{(K_\omega)^{5/6}-1}\\
&=& \frac{k_2}{(K_\omega)^{5/6}-1} \quad \text{where $k_2= \frac{J^{N_3}(c_0)K^\star}{k_1^2}>0$.}\\
 \end{eqnarray*}
 \item
\begin{eqnarray*}
J^{N_3}(c_0)D_{N_3}(c_0) &\overset{\text{Definition}}{=}& \frac{J^{N_3}(c_0) }{\sqrt{K_\omega}} \left[ \sum_{i=0}^{N_3-1}  \frac{1}{d_i(c_0)} \right]^{-1}\\
 &\overset{\text{Definition}}{=}& \frac{J^{N_3}(c_0) }{\sqrt{K_\omega}} \left[ \sum_{i=0}^{N_3-1}  \frac{J^i(c_0)}{dist(c_i,\mathcal{C})dist(c_i,\mathcal{S})} \right]^{-1}\\
  &\overset{\text{\text{Item } (1)}}{\geq} & \frac{J^{N_3}(c_0) }{\sqrt{K_\omega}}   \frac{(K_\omega)^{5/6}-1} {k_2}\\
  &\overset{K_\omega \text{ large}}\approx & \frac{J^{N_3}(c_0) }{\sqrt{K_\omega}}   \frac{(K_\omega)^{5/6}} {k_2}\\
    &= & {J^{N_3}(c_0) }   \frac{(K_\omega)^{1/3}} {k_2}\\
    &= & k_3  {(K_\omega)^{1/3}} \quad \text{where $k_3= \frac{J^{N_3}(c_0) }{k_2}>0$.}\\
 \end{eqnarray*}

%\item   Since
%\begin{eqnarray*}
%\frac{1}{\sqrt{K_\omega}}\frac{d_{N_1}(c_0)}{D_{N_1}(c_0)}&\overset{Def.}{=} &\sum_{i=0}^{N_1-1} \frac{J^i(c_0)}{J^{N_1}(c_0)}\frac{d_\mathcal{C}(c_{N_1})d_\mathcal{S}(c_{N_1})}{d_\mathcal{C}(c_i)d_\mathcal{S}(c_i)} \\
%&\leq &\sum_{i=0}^{N_1-1} \frac{J^i(c_0)}{\left(\frac{K_\omega \sigma}{K_0}\right)^{N_1-i}J^i(c_0)}\frac{d_\mathcal{C}(c_{N_1})d_\mathcal{S}(c_{N_1})}{d_\mathcal{C}(c_i)d_\mathcal{S}(c_i)} \\
%&\leq &\sum_{i=0}^{N_1-1} \frac{J^i(c_0)}{\left(\frac{K_\omega \sigma}{K_0}\right)^{N_1-i}J^i(c_0)}\frac{d_\mathcal{C}(c_{N_1})d_\mathcal{S}(c_{N_1})}{k_1} \\
%&\overset{Lemma \ref{aux1}}\leq&  \sum_{i=0}^{N_1-1} \frac{2}{((K_\omega)^{\frac{5}{6}}/K_0)^{N_1-i}(K_\omega)^{\frac{-1}{6}} \varepsilon_0}\\ &\leq& \frac{1}{{K_\omega}},
%\end{eqnarray*}
 \item 
\begin{eqnarray*}
\left|h_a^{N_3+1}(I_{N_3}(c))\right |  &\geq &  J^{N_3}(h_a(c))\, \, .\, \, |h_a(I_{N_3}(c))| \\ \\
&\overset{\text{Definition of $I_p$}}{\geq} &\frac{1}{2}J^{N_3}(h_a(c))\, \, .\, |D_{N_3}(h_a(c))-D_{N_3+1}(h_a(c))|\\ \\
 & \geq&\frac{1}{4}J^{N_3}(c_0)|D_{N_3}(h_a(c))|\\ \\
 &\overset{h_a(c)=c_0}=&\frac{1}{4}J^{N_3}(c_0)|D_{N_3}(c_0)|\\ \\
 &  \overset{\text{Item } (2)}{\geq}& k_4(K_\omega)^{1/3} \quad \text{where $k_4= k_3/4>0$.}\\
\end{eqnarray*}
 \end{enumerate}
 If $K_\omega$ is large and $a\in \Delta$, then  $\left|h_a^{N_3+1}(I_{N_3}(c))\right | > k_4(K_\omega)^{1/3} \gg 2\pi.$ This finishes the proof. 

\end{proof}

\begin{remark}
\label{final1}   
 As $K_\omega$ gets larger, the contracting regions get smaller and the dynamics is more and more expanding in most of the phase space. The recurrence to $\mathcal{S}$ is inevitable.
\end{remark}

 \section{Proof of Theorem \ref{prop_main}: ``large'' strange attractors}
  \label{Prova Th A1}
 
The proof of Theorem \ref{prop_main} needs the results of Sections \ref{proof Th B},  \ref{s:results}   and \ref{s:unbounded}. First of all, note that there is a correspondence between the set $\Delta$ of Section 7 and the set $\Delta_\lambda$ of statement of Theorem~\ref{prop_main}: 

$$a=-K_\omega \ln \lambda \pmod{2\pi} \quad 
\Leftrightarrow \quad \lambda= \exp\left(\frac{a-2k\pi}{K_\omega}\right), \quad \text{where} \quad k\in \NN.$$
 
 For $\lambda\in \Delta_\lambda$, let $B\subset \mathcal{U}$ a small ball centered at $x_0\in \mathcal{U}$ and radius $\varepsilon>0$ and define $$B^\star = B \backslash W^s(P_1).$$ It is clear that $Leb_2(B)= Leb_2(B^\star)>0$ since $Leb_2(W^s(P_1))=0$. The $\omega$-limit of $B^\star$ is the set  $W^u(P_2)$, parametrized by $y=0$, whose dynamics is governed by the map $h_a$ of Lemma \ref{important lemma}. If $K_\omega$ is large enough, then:\\
 \begin{enumerate}
 \item there exists a subset of $B^\star$ such that, after a finite number of iterations, its first component contains an interval  with a critical point of $h_a$ whose orbit has a positive Lyapunov exponent (combination  of Item (2) of Lemma \ref{lemma7.2} and of Lemma \ref{union1}). Indeed,
 $$
 \lambda (h_a,c)= \lim_{n\in \NN} \frac{1}{n} |(h_a^n)'(h_a(c))| \geq \frac{1}{1000} \ln (K_\omega)>0.\\
 $$ 
 
 \item  Lemma \ref{lema auxiliar} says that there exists a subset of the projection of $B^\star$ whose $\omega$-limit covers  the entire $\EU^1$. More precisely, by   item (3) of Lemma \ref{lema auxiliar}, we have:   $$\left|h_a^{N_3+1}(I_{N_3}(c))\right|    \gg 2\pi.$$
 \end{enumerate}
 This finishes the proof of Theorem \ref{prop_main}.

 \section{Proof of Theorem \ref{thm:C}: abundant infinite switching}
 \label{Prova Th B1}
 Let $\Gamma_\lambda$ be the heteroclinic network defined in Section \ref{main results},  $\lambda\in \Delta_\lambda$ and $K_\omega$ sufficiently large. Fix an infinite admissible path $\sigma^\infty $ on $\Gamma_\lambda$ and let $V_\Gamma$ be any small tubular neighbourhood of $\Gamma_\lambda$. Any small ball (open set) $B$ contained in $V_\Gamma\cap \Out(P_2)$ shadows $W^u(P_2)\cap \Out(P_2)$, after a finite number of iterations, as a consequence of Theorem \ref{prop_main}. This means that there exists a subset of $B\backslash W^s(P_1)$ that spreads its solutions  around all possible connections leaving $P_2$ and $P_1$.  This  proves infinite switching. The phenomenon is realized by all initial conditions lying in $
B^\star \cap \mathcal{G}_\lambda^n (V_\Gamma \cap \Out(P_2)),$   $n\in \NN$,  which has positive Lebesgue measure. Within any small open ball near the network, there exists a set of initial conditions with positive Lebesgue measure shadowing any prescribed infinite path.  
   %This set shadow the connections $\gamma_1$ and $\gamma_2$ and the two transverse connections.

\section{An example}
\label{s:example}
This research article has been motivated by the following example introduced in  \cite{RodLab}. Some preliminaries about symmetries of a vector field may be found in  \cite{Kirk2010, RodLab}.
 For $\lambda \in \,[0,1]$, our object of study is the one-parameter family of vector fields on $\RR^{4}$ $$x=(x_1,x_2,x_3,x_4)\in\RR^4 \quad \mapsto \quad  g_{\lambda}(x)$$
defined for each $x=(x_1,x_2,x_3,x_4)\in \RR^4$ by
\begin{equation}\label{example}
 \left\{
\begin{array}{l}
\dot{x}_{1}=x_{1}(1-r^2)- {\omega }x_2-\alpha_1x_1x_4+\alpha_2x_1x_4^2  \\
\dot{x}_{2}=x_{2}(1-r^2)+ {\omega } x_1-\alpha_1x_2x_4+\alpha_2x_2x_4^2 \\
\dot{x}_{3}=x_{3}(1-r^2)+\alpha_1x_3x_4+\alpha_2x_3x_4^2+\lambda x_1x_2x_4 \\
\dot{x}_{4}=x_{4}(1-r^2)-\alpha_1(x_3^2-x_1^2-x_2^2)-\alpha_2x_4(x_1^2+x_2^2+x_3^2)-\lambda x_1x_2x_3 \\
\end{array}
\right.
\end{equation}
where $\dpt \dot{x}_i=\frac{\partial x_i}{\partial t},$  $r^2=x_{1}^{2}+x_{2}^{2}+x_{3}^{2}+x_{4}^{2}$, and

$$
\omega>0, \qquad \beta <0<\alpha, \qquad \beta^2<8 \alpha^2 \qquad \text{and} \qquad |\beta|<|\alpha|.
$$
The unit sphere $\EU^3\subset\RR^{4}$ is invariant under the corresponding flow and every trajectory with nonzero initial condition is forward asymptotic to it (cf. \cite{RodLab}). Indeed, if $\left\langle .\,, . \right\rangle $ denotes the usual inner product in $\RR^4$, then it is easy to check that:
\begin{lemma}
For every $x \in \EU^3$ and  $\lambda \in [0,1]$, we have $\left\langle g_\lambda(x), x\right\rangle =0$.
\end{lemma}
  
The origin is repelling since all eigenvalues of $Dg_{\lambda}(0,0,0,0)$   have positive real part, where $\lambda \in [0,1]$.   
The vector field $g_0$ is equivariant under the action of the compact Lie group $\mathbb{SO}(2)(\gamma_\psi)\oplus \ZZ_2(\gamma_2)$, where $\mathbb{SO}(2)(\gamma_\psi)$ and $\ZZ_2(\gamma_2)$ act on $\RR^4$ as
$$\gamma_\psi(x_1, x_2,x_3,x_4)=(x_1\cos \psi -x_2 \sin \psi, x_1\sin \psi +x_2\cos \psi, x_3,x_4), \quad \psi \in [0, 2\pi] $$
given by a phase shift $\theta \mapsto \theta+ \psi$ in the first two coordinates, and 
$$ \gamma_2(x_1, x_2,x_3,x_4)=(x_1, x_2,-x_3,x_4).$$
By construction, $\lambda$ is the controlling parameter of the $\ZZ_2(\gamma_2)-$symmetry breaking  but keeping the $\mathbb{SO}(2)(\gamma_\pi)$--symmetry, where
$$
\gamma_\pi (x_1, x_2, x_3, x_4)=(-x_1, -x_2, x_3, x_4). $$
 When restricted to the sphere $\EU^3$,  the flow of $g_0$ has %four repelling equilibria  and
 two equilibria 
$$P_1 =(0,0,0,+1) \quad \quad \text{and} \quad \quad P_2 = (0,0,0,-1), $$
which are hyperbolic saddle-foci. 
The linearization of $g_0$ at $P_1$ and $P_2$ has eigenvalues
$$ -(\alpha-\beta) \pm \omega i, \,\,  \alpha+\beta \qquad \text{and} \qquad (\alpha + \beta)\pm \omega i, \,\,  -(\alpha-\beta)$$
respectively. Using the terminology of Section \ref{s:setting}, we get: 
\begin{equation}
\label{constants}
C_1 = C_2 = \alpha - \beta>0, \qquad E_1 = E_2 = \alpha + \beta>0, \qquad \delta_1= \delta_2 = \frac{\alpha - \beta}{\alpha + \beta }>1
 \end{equation}
and
\begin{equation}
\label{constants3}
   K_\omega= \frac{2 \alpha \omega  }{(\alpha +\beta)^2}>0.
 \end{equation}

The 1D-connections are contained in: 
\begin{eqnarray*}
\overline{W^u(P_1)} \cap \EU^3&=& \overline{W^s(P_2)} \cap \EU^3=\text{Fix}(\mathbb{SO}(2)(\gamma_\psi))\cap \EU^3\\ 
&=& \{(x_1,x_2,x_3,x_4): \quad x_1=x_2=0, \quad x_3^2 + x_4^2 = 1\}
\end{eqnarray*}
and the 2D-connection is contained in
\begin{eqnarray*}
\overline{W^u(P_2)} \cap \EU^3&=& \overline{W^s(P_1)} \cap \EU^3=\text{Fix}(\ZZ_2(\gamma_2))\cap \EU^3 \\
&=& \{(x_1,x_2,x_3,x_4): \quad x_1^2 + x_2^2 + x_4^2 = 1, \quad x_3=0\} .
\end{eqnarray*}
 
  \begin{figure}[h]
\begin{center}
\includegraphics[height=9.0cm]{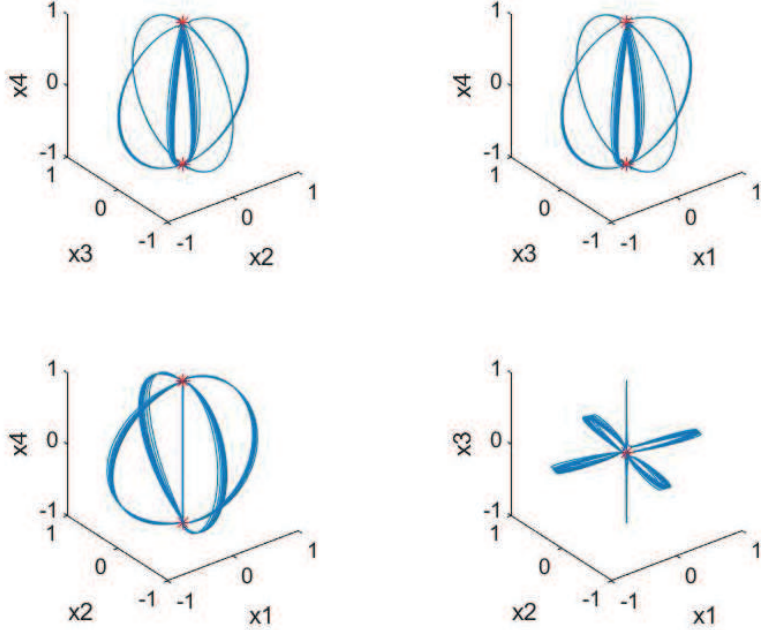}%\\\includegraphics[height=10.0cm]{Numerics2}
\end{center}
\caption{\small Projection into three coordinates of the solution of  (\ref{example})  with initial condition $(0.01; 0.01; 0.01; 1)$ near $W^u(P_2)$, with $\lambda=0.1$,   {$\omega=1$}, $\alpha=1$, $\beta= -0.1$, and $t\in [0,10000]$. Red stars represent the equilibria $P_1$ and $P_2$. The greatest Lyapunov exponent associated to that trajectory is $0.0004 \gtrsim 0$. } 
\label{Numerics1}
\end{figure}

 \begin{figure}[h]
\begin{center}
\includegraphics[height=9.0cm]{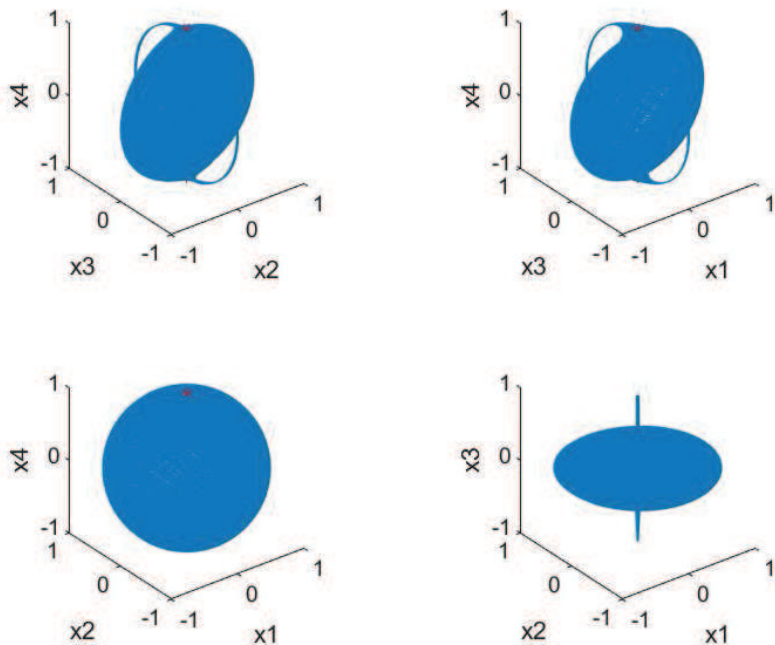}%\\\includegraphics[height=10.0cm]{Numerics2}
\end{center}
\caption{\small Projection into three coordinates of the solution of  (\ref{example})  with initial condition $(0.01; 0.01; 0.01; 1)$ near $W^u(P_2)$, with $\lambda=0.1$,  {$\omega=10$}, $\alpha=1$, $\beta= -0.1$, and $t\in [0,10000]$. The greatest Lyapunov exponent associated to that trajectory is $0.1309>0$. } 
\label{Numerics2}
\end{figure}

The two-dimensional invariant manifolds of $P_1$ and $P_2$ are contained in the two-sphere $\text{Fix}(\ZZ_2(\gamma_2 ))\,\cap\, \EU^3.$ It is precisely the symmetry $\ZZ_2(\gamma_2)$ that forces the two-invariant manifolds $W^u(P_2)$ and $W^s(P_1)$ to coincide. We denote by $\Gamma$ the \emph{heteroclinic network} formed by the two equilibria, the two one-dimensional connections by $[P_1 \rightarrow P_2]$ and the sphere by $[P_2 \rightarrow P_1]$.  By the way the vector field of \eqref{example} is constructed, the equilibria $P_1$ and $P_2$ have the same \emph{chirality}. Therefore:
\begin{lemma}
If $\lambda=0$,  the flow of \eqref{example} satisfies \textbf{(H1)--(H5)} described in Section \ref{ss:oc}. 
\end{lemma}
For $\lambda=0$, the flow of \eqref{example} exhibits an asymptotically stable heteroclinic network $\Gamma$ associated to $P_1$ and $P_2$. 
Since all the heteroclinic connections are contained in fixed point subspaces, the sphere $W^u(P_2)$ acts as a barrier and there is no switching near $\Gamma$.
The parameter $\lambda$   plays the role of  \textbf{(H6)--(H7)}, after possible rescaling.

 \begin{lemma}\cite[Appendix A]{RodLab}
\label{torus_cor}
For $\lambda>0$ small, the following conditions hold:
\begin{enumerate} 
\item   $W^u(P_2)$ and $W^s(P_1)$ intersect transversely. 
\item there are  two one-dimensional connections from $P_1$ to $P_2$.
\end{enumerate}
\end{lemma}
    When $\lambda>0$, although we break the $\mathbb{SO}(2)(\gamma_\psi)$--equivariance, the $\ZZ_2(\gamma_\pi)$--symmetry is preserved. This is why the connections lying in  $x_1=x_2=0$ persist.
 For $\lambda>0$, let us denote by $\Gamma_\lambda$ the emerging heteroclinic network (with a finite number of connections from $P_2$ to $P_1$) and $
 \mathcal{U}$ a small absorbing domain of $\Gamma$.   As a consequence of Theorems \ref{prop_main} and  \ref{thm:C}, we may easily conclude that:\\
\begin{corollary}
  With respect to the dynamics of \eqref{example}, there exists $\omega^0 \gg 1$ such that if $\omega>\omega^0$, then: \\
  \begin{enumerate}
  \item  there exists a set $ \Delta_\lambda\subset [0, \lambda_0]$ ($\lambda_0$ small)  with positive Lebesgue measure such that if $\lambda \in \Delta_\lambda$, then the flow of $g_\lambda$  contains a ``large'' strange attractor; \\
  \item  there exists a set $ \Delta_\lambda \subset [0, \lambda_0]$  ($\lambda_0$ small)  with positive Lebesgue measure such that if $\lambda \in \Delta_\lambda$, then  the network $\Gamma_\lambda$  exhibits  abundant infinite switching. \\ 
   \end{enumerate}
   Within any small ball within $\mathcal{U}$,  there exists a set of initial conditions with positive Lebesgue measure shadowing any prescribed infinite path.   \\
\end{corollary}
\bigbreak
\begin{remark}
The generic Hypothesis \textbf{(H8)} for model (\ref{example}) is a technical point impossible to be rigorously checked. 
\end{remark}
 
\bigbreak
 Numerical simulations of (\ref{example}) in Figures \ref{Numerics1} and \ref{Numerics2} for  $\lambda>0$ suggest the existence of strange attractors and the effect of the parameter $\omega$.
   We can observe that if $\omega=10$ (large), then the non-wandering set associated to the initial condition $(0.01; 0.01; 0.01; 1)$ covers the sphere $W^u(P_2)$. This set cannot attract open sets of initial conditions (because $W^s(P_1)$ is ubiquitous in the absorbing domain), although it attracts sets with positive Lebesgue measure.   Solutions seem to visit the connections leaving $P_2$ in a uniformly distributed manner.

Since $W^u(P_2)$ plays an essential role in the construction of the  strange attractor, we chose  the initial condition $(0.01; 0.01; 0.01; 1)$ close to $W^u (P_2)$ to collect the main dynamical properties of the maximal attracting set of \eqref{example}.

The computer experiments of Figures \ref{Numerics1} and \ref{Numerics2} have been performed using \emph{Matlab (R2021b, Mathworks, Natick, MA, USA)}, using the method  described in Section 5.3 of \cite{CastroR2020}. The Lyapunov exponents computation for the present work is based on the algorithm proposed by Wolf \cite{Wolf} adapted from the freely available Matlab functions \footnote{https://www.mathworks.com/matlabcentral/fileexchange/4628-calculation-lyapunov-exponents-for-ode, MATLAB Central File Exchange. Retrieved February 3, 2023.}. 
The input parameters of  \emph{Lyapunov.m} function include the number of equations, the start and end values for time, the time step, the initial condition, and a handle of function with right-hand side of the extended ODE-system, coupled with a variational equation. The \emph{ode45.m} function, based on the Runge-Kutta algorithm, was used for the Ordinary Differential Equations integration, with a  step size of $0.5$. This function uses relative and absolute error tolerances of $1\times10^{-3}$ and $1\times10^{-6}$, respectively. The absolute error is a threshold below which the value of the solution becomes unimportant and the other is an error relative to the magnitude of each solution component. These functions have been adapted for the computation of the Lyapunov exponents of \eqref{example}.

%The computer experiments of Figures \ref{Numerics1} and \ref{Numerics2} have been performed using \emph{Matlab (R2021b, Mathworks, Natick, MA, USA)}, 
%To estimate the  Lyapunov exponents, we have used the algorithm for differential equations introduced in [Wolf et al., 1985] with a Taylor series integrator. The ode45.m function, based on the Runge-Kutta algorithm, was used for the Ordinary Differential Equations integration, with an integration step size of $0.5$. This function uses relative and absolute error tolerances of $1×10^{-3}$ and $1×10^{-6}$, respectively. The absolute error is a threshold below which the value of the solution becomes unimportant and the other is an error relative to the magnitude of each solution component.

\section{Discussion and concluding remarks} 
\label{s:discussion}

 This paper finishes the discussion about the dynamics of the class of examples  presented in \cite{LR, LR2016} and numerically explored in \cite{RodLab}. 
 Our starting point is a one-parameter family of ordinary 
 differential equations  (\ref{general2.1}) defined in the unit sphere $\EU^3$
with two saddle-foci whose organising  center ($\lambda=0$) shares all the invariant manifolds, forming an attracting heteroclinic network $\Gamma$.  When $\lambda\neq 0$, we assume that the one-dimensional connections persist, and the  two dimensional invariant manifolds intersect transversely, forming a \emph{heteroclinic tangle}.
  The existence of  infinite switching near $\Gamma_\lambda$ follows the reasoning of \cite{ALR} but the Lebesgue measure of the set of initial conditions realising switching was unknown. The main contribution of this paper is twofold: \\
 \begin{enumerate}
 \item  First, in Theorem \ref{prop_main}, we prove that, if the \emph{twisting number} $K_\omega$  is large enough (see \eqref{constants2}), then for a subset of  $  [0, \lambda_0]$ ($\lambda_0>0$ is small)   with positive Lebesgue measure, the dynamics of the first return map $\mathcal{G}_\lambda$ exhibits non-uniformly hyperbolic strange attractors winding around the  annulus   $\Out(P_2)$  -- \emph{i.e.} there exist ``large'' strange attractors in the terminology of \cite{BST98}. The $\omega$-limit  of almost all points in $\mathcal{U}$  (absorbing domain of $\Gamma$) contains $W^u(P_2)$.  These strange attractors have one positive Lyapunov exponent, are non-uniformly expanding, are not robustly transitive, and  coexist (in the phase space) with infinitely many heteroclinic connections found in \cite{LR2016}.   \\
 \item Secondly, the proof of ``large'' strange attractors allows us to prove Theorem  \ref{thm:C}:  the $\omega$-limit of any small ball in $\mathcal{U}$ contains the whole set $W^u(P_2)\cap \Out(P_2)$,  allowing the shadowing of any infinite path.  The original network structure will be observed in the long term dynamics even though   $\Gamma_\lambda$ is not attracting. \\
 \end{enumerate}
  
   These dynamical phenomena are caused essentially by three main ingredients: the existence of saddle-foci in the network, the transverse intersection of $W^u(P_2)$ and $W^s(P_1)$  and the fact that they unfold a coincidence at $\lambda=0$. In the example discussed in Section \ref{s:example}, this coincidence is caused by the $\mathbb{SO}(2)$--equivariance where    $\omega>0$ is the unique parameter that matter to prompt the birth of ``large'' strange attractors.

      Our findings have the same flavour to those of \cite{Rodrigues_2022_DCDS}, even though Hypothesis \textbf{(H7)} is different and the classical theory of \cite{WY} does not hold in the case under consideration.    As far as we know,  there is no analogue of the effect of singularities of the singular limit  in previous studies about infinite switching.

  The singular family associated to $\mathcal{G}_\lambda$ has singularities with unbounded derivative. This fact creates expansion and chaos. 
 We have made use of the critical interval     constructed by \cite{Takahasi1} to realize that, after a finite number of iterations of the singular limit $h_a$, any small ball of initial conditions in $\Out(P_2)$ will cover the whole $W^u(P_2)$ infinitely many times.  
 
  The idea of abundant infinite switching is also implicit in Section 3.2 of \cite{Kirk2010}; if $\omega>0$ is sufficiently large, then  the range of the angular coordinate   covers  $[0,2\pi]$. See also Figure 5 of \cite{Kirk2010} where the   length of the line segment is chosen in such a way that its image covers the full range of values of the angular component.

This article is part of a systematic study of bifurcations of Bykov cycles   and finishes the generic study of their  dynamical properties. 
The next natural  problem   is the study of statistical and ergodic properties of this class of examples, whose singular cycle has an attracting component and another with an unbounded derivative. Since the classical work by \cite{WY} does not hold in this class of examples, its adaptation is a big challenge. 
For $\lambda \in \, \, (0, \lambda_0]\backslash \Delta_\lambda$ ($\lambda_0$ small), we would like to investigate the topological entropy associated associated to \eqref{general2.1}.     We defer this task to  a future work.

 \section*{Acknowledgments}

The authors are grateful to the two referees for the constructive comments, corrections  and suggestions which helped to improve the readability of this manuscript.

\end{document}